\documentclass[10 pt,a4paper]{article}

\usepackage{latexsym,amsfonts,amsmath,amsthm,amssymb, graphicx, mathrsfs, fancyhdr, makeidx, amscd,  color, mathdesign, fontenc, stix, comment}
\usepackage[all]{xy}

\newtheorem{theorem}{Theorem}[section]

\newtheorem{definition}[theorem]{Definition}
\newtheorem{example}[theorem]{Example}

\newtheorem{notation}[theorem]{Notation}

\newtheorem{remark}[theorem]{Remark}

\newcommand{\R}{\mathbb{R}}

\newcommand{\metric}{\langle \, , \, \rangle}

\newcommand{\disp}{\displaystyle}

\newcommand{\ra}{\rightarrow}

\newcommand{\eps}{\varepsilon}

\newcommand{\II}{\mathrm{II}}
\newcommand{\Sph}{\mathbb{S}}
\newcommand{\di}{\mathrm{d}}

\newcommand{\HH}{\mathbb{H}}

\newcommand{\Ricc}{\mathrm{Ric}}
\newcommand{\cut}{\mathrm{cut}}
\newcommand{\vol}{\mathrm{vol}}
\newcommand{\lip}{\mathrm{Lip}}
\newcommand{\loc}{\mathrm{loc}}

\newcommand{\capac}{\mathrm{cap}}

\newcommand{\LL}{\mathscr{L}}
\newcommand{\Sect}{\mathrm{Sec}}
\newcommand{\BB}{\mathbb{B}}
\newcommand{\smp}{(\mathrm{SMP}_\infty)}
\newcommand{\wmp}{(\mathrm{WMP}_\infty)}

\newcommand{\RR}{\mathbb{R}}

\DeclareMathOperator{\diver}{div\,}

\begin{document}

\author{Bruno Bianchini $\, ^{(1)}$ \and Giulio Colombo $\, ^{(2)}$ \and Marco Magliaro $\, ^{(3)}$ \and Luciano Mari $\, ^{(4)}$ \and Patrizia Pucci $\, ^{(5)}$ \and Marco Rigoli $\, ^{(2)}$}
\title{\textbf{Recent rigidity results for graphs with prescribed mean curvature}}
\date{}
\maketitle

\scriptsize \begin{center} $(1) \ \ $ Dipartimento di Matematica Pura e Applicata, Universit\`a degli Studi di Padova\\
Via Trieste 63, I-35121 Padova (Italy). E-mail: bianchini@dmsa.unipd.it
\end{center}

\scriptsize \begin{center} $(2) \ \ $ Dipartimento di Matematica,
Universit\`a
degli studi di Milano,\\
Via Saldini 50, I-20133 Milano (Italy). E-mail: giulio.colombo@unimi.it, marco.rigoli55@gmail.com
\end{center}

\scriptsize \begin{center} $(3) \ \ $ Departamento de Matem\'atica, Universidade Federal do Cear\'a\\
Campus do Pici -Bloco 914, Av. Humberto Monte s/n,  60.455-760 Fortaleza (Brazil). E-mail: marco.magliaro@gmail.com
\end{center}

\scriptsize \begin{center} $(4) \ \ $ Dipartimento di Matematica, Universit\'a degli Studi di Torino,\\
Via Carlo Alberto 10, I-10123 Torino (Italy). E-mail: luciano.mari@unito.it
\end{center}

\scriptsize \begin{center} $(5) \ \ $ Dipartimento di Matematica e Informatica, Universit\`a degli Studi di Perugia\\
Via Vanvitelli 1, I-06123 Perugia (Italy). E-mail: patrizia.pucci@unipg.it
\end{center}

\normalsize

\maketitle

\begin{center}
To Alberto Farina, on the occasion of his 50th birthday, with great esteem.
\end{center}

\tableofcontents


\begin{abstract}
	This survey describes some recent rigidity results obtained by the authors for the prescribed mean curvature problem on  graphs $u : M \ra \R$. Emphasis is put on minimal, CMC and capillary graphs, as well as on graphical solitons for the mean curvature flow, in warped product ambient spaces. A detailed analysis of the mean curvature operator is given, focusing on maximum principles at infinity, Liouville properties, gradient estimates. Among the geometric applications, we mention the Bernstein theorem for positive entire minimal graphs on manifolds with non-negative Ricci curvature, and a splitting theorem for capillary graphs over an unbounded domain $\Omega \subset M$, namely, for CMC graphs satisfying an overdetermined boundary condition.
\end{abstract}

\section{Introduction}


The purpose of this survey is to illustrate some recent rigidity results, obtained by the authors, about the prescribed mean curvature problem on a complete Riemannian manifold $M$. It is a pleasure to dedicate our work to Alberto Farina on the occasion of his birthday, and to have the opportunity to discuss some of his significant contributions to the theory. Hereafter, we shall always restrict to real valued functions $u: M \ra \R$, and their associated entire graphical hypersurfaces $\Sigma \hookrightarrow \R \times M$ described by the map 
\[
\varphi : M \ra \R \times M, \qquad \varphi(x) = \big( u(x),x\big)
\]
(throughout the paper, the word \emph{entire} will mean defined on all of $M$). Taking into account the huge literature on hypersurfaces with prescribed mean curvature, notably on minimal and constant mean curvature (CMC) ones, some choices have to be made, and we decided to exclusively focus on references considering the graphical case in codimension $1$. Thus, unfortunately, important works concerning embedded (immersed) hypersurfaces have been omitted, unless they provide new results in the graphical case too. Concerning the prescribed mean curvature problem in Euclidean space, that will be briefly recalled in the next section, we recommend the surveys \cite{simon,farina_minimal} for an excellent account. Simon's \cite{simon} focuses on the minimal surface equation, while Farina's \cite{farina_minimal} includes Liouville theorems for more general operators.\par

\subsection{The prescribed curvature problem in $\R \times M$}

To put our results into perspective, we briefly recall the main achievements concerning the problem on $\R^m$. We begin with entire minimal graphs in $\R^{m+1}$, described by functions $u : \R^m \ra \R$ solving the minimal surface equation
\begin{equation}\label{eq_minimal_Rm}
\diver \left( \frac{Du}{\sqrt{1+|Du|^2}} \right) = 0,
\end{equation}
that is the Euler-Lagrange equation of the area functional 
\begin{equation}\label{eq_functional_min}
\Omega \Subset \R^m \mapsto \int_\Omega \sqrt{1+|Du|^2}\di x
\end{equation}
with respect to compactly supported variations. In 1915, Bernstein \cite{bernstein} proved that the only solutions of \eqref{eq_minimal_Rm} on $\R^2$ are affine functions (his proof, highly nontrivial, has been perfected in \cite{hopf_bern,mickle}, where a missing detail was corrected). Since then, various other proofs appeared in the literature, each one exploiting some different peculiarities of the two dimensional setting, especially: the possibility to use holomorphic function theory and the uniformization theorem \cite{nitsche,heinz_2,fcs,docarmo_peng}, or the quadratic volume growth of $\R^2$ coupled with the specific form of the Gauss-Bonnet theorem \cite{pogorelov,simon,mt}. A presentation of some of them can be found in \cite{farina_minimal}. The problem whether Bernstein's rigidity theorem holds for minimal graphs over $\R^m$ if $m \ge 3$ challenged researchers for long, and the efforts to answer this question gave rise to many new and deep techniques in Geometric Analysis. The first proof of Bernstein theorem that allowed for a possible higher dimensional generalization was given by Fleming \cite{fleming}, and since then, in few years, the problem has been completely solved affirmatively in dimension $m \le 7$ by works of De Giorgi (\cite{dg1,dg2}, $m=3$), Almgren (\cite{almgren}, $m=4$) and Simons (\cite{simons}, $m \le 7$), while Bombieri, De Giorgi and Giusti \cite{bdgg} constructed the first counterexample if $m \ge 8$ (a large class of further examples was later given by Simon \cite{simon_jdg}). Summarizing, the following theorem holds: 
\begin{equation}\label{bernstein_1}\tag{$\mathscr{B}1$}
\text{each minimal graph over $\R^m$ is affine} \qquad \Longleftrightarrow \qquad m \le 7.	 
\end{equation}
If $u$ satisfies some a-priori bounds, then more rigidity is expected. Indeed, as a consequence of the local  gradient estimate for minimal graphs $u : B_{r}(x) \subset \R^m \ra \R$ due to Bombieri, De Giorgi and Miranda \cite{bdgm}:
\[
|Du(x)| \le c_1\exp\left\{ c_2 \frac{u(x)- \inf_{B_r}u}{r} \right\}, 
\]
for some constants $c_j(m)$, one deduces both that
\begin{equation}\label{bernstein_2}\tag{$\mathscr{B}2$} 
\text{positive minimal graphs on $\R^m$ are constant, for each $m \ge 2$}
\end{equation}
and that
\begin{equation}\label{bernstein_3}\tag{$\mathscr{B}3$}
\begin{array}{c}	
\text{minimal graphs on $\R^m$ with at most linear growth have bounded gradient,} \\ 
\text{thus, by work of Moser \cite{jmoser} and De Giorgi \cite{dg1,dg2}, they are affine.}
\end{array}
\end{equation}		
Further splitting results in the spirit of \eqref{bernstein_3} can be obtained by just assuming the boundedness of some of the directional derivatives of $u$. In particular, improving on work of Bombieri and Giusti \cite{bg}, Farina in \cite{far1,far2} showed that $u$ is affine provided that $n-7$ partial derivatives of $u$ are bounded on one side.\par  
The possible existence of a similar results for entire minimal graphs defined on more general complete Riemannian manifolds $(M^m,\sigma)$ heavily depends on the geometry of $M$. If $\R \times M$ is endowed with the product metric 
\[ 
\metric = \di t^2 + \sigma
\]
then the (normalized) mean curvature $H$ of the image $\varphi(M)$ in the upward pointing direction is given by 
\[
mH = \diver \left( \frac{Du}{\sqrt{1+|Du|^2}} \right), 
\] 	
formally the same as for $\R^m$, where now $D, \diver$ are the gradient and divergence in $(M, \sigma)$. In particular, entire minimal graphs solve \eqref{eq_minimal_Rm} on $M$. However, for instance if $M$ is the hyperbolic space $\HH^m$, \eqref{eq_minimal_Rm} admits uncountably many bounded solutions. Indeed, the Plateau problem at infinity is solvable for every continuous $\phi$ on the boundary at infinity $\partial_\infty \HH^m$, that is, there exists $u$ solving \eqref{eq_minimal_Rm} on $\HH^m$ and approaching $\phi$ at infinity. The result was shown in \cite{nr} for $m=2$, while the higher dimensional case, to our knowledge, first appeared in 2010 as a special case of \cite{ES_F_R}. We note in passing that, as a consequence of Theorem \ref{teo_graph_soloricci_intro} below, any solution of Plateau's problem on $\HH^m$ has a bounded gradient.\par
Recently, various interesting works	investigated the Plateau's problem at infinity on Cartan-Hadamard (CH) manifolds $M$, that is, on complete, simply connected manifolds with non-positive sectional curvature $\Sect$, elucidating the sharp thresholds of $\Sect$ to guarantee its solvability or not (cf. \cite{rite_1,chr,chr2,chh,hr,chh_nonexistence}). The picture is quite subtle and an updated account can be found in \cite{esko_survey}. In particular, having denoted with $r$ the distance from a fixed origin $o \in M$, we mention that the Plateau problem at infinity on such $M$ is solvable whenever any of the following pinching conditions hold outside of a compact set $B_R(o)$: 
\[
\begin{array}{ll}
\disp -\frac{e^{2 \kappa r}}{r^{2+\eps}} \le \Sect \le - \kappa^2 & \quad \text{(by \cite{rite_1}, see also \cite[Cor 1.7]{chr}),} \\[0.3cm]
\disp - r^{2(\kappa'-2)-\eps} \le \Sect \le - \frac{\kappa'(\kappa'-1)}{r^2} & \quad \text{(by \cite[Thm. 1.5]{chr}),}
\end{array}
\]
for constants $\eps, \kappa>0$, $\kappa'>1$. Hereafter, an inequality $\Sect \le G(r)$ will shortly mean 
\[
\Sect(\pi_x) \le G\big(r(x)\big) \qquad \text{for every $x \in M$ and $2$-plane $\pi_x \le T_xM$.}
\] 	
In general, a two-sided bound on $\Sect$ is necessary, although a (suitable) upper bound suffices on rotationally symmetric CH manifolds (cf. \cite{chh_nonexistence}). Indeed, in \cite{hr} the authors constructed a $3$-dimensional CH manifold with $\Sect \le -\kappa^2<0$, that admits bounded solutions of the minimal surface equation, but for which the Dirichlet problem at infinity is not solvable for \emph{any} nonconstant boundary datum $\phi$. The dimension restriction is not removable, since for $m=2$ the one-sided condition $\Sect \le - \kappa^2 < 0$ suffices for the solvability of the Plateau problem at infinity, see \cite{galvezrosenberg}. We also mention the surprising existence of parabolic, entire minimal graphs in $M^2 \times \R$, with $M^2$ a CH surface satisfying $\Sect \le - \kappa^2 < 0$, obtained in \cite{collinrosenberg, galvezrosenberg}.\par
It is therefore tempting to ask whether results similar to those for $\R^m$ may hold on manifolds with non-negative curvature, and in view of the structure theory initiated by Cheeger and Colding in  \cite{cc_almost_rigidity,cc_1,cc_2,cc_3}, a particularly intriguing class to investigate in this respect is that of manifolds with $\Ricc \ge 0$. This is enforced by the analogy with the behaviour of harmonic functions on $M$. In fact, positive harmonic functions on a manifold with $\Ricc \ge 0$ are constant \cite{yau, chengyau}, and harmonic functions with linear growth provide splitting directions for the tangent cones at infinity of $M$, and for $M$ itself provided that $\Ricc \ge 0$ is strengthened to $\Sect \ge 0$ \cite{ccm}. The parallelism is even more evident by considering that \eqref{eq_minimal_Rm} rewrites as the harmonicity of $u$ on $\Sigma$ with its induced metric.\par
While we are aware of no references considering the full Bernstein theorem \eqref{bernstein_1} on spaces different from $\R^m$, and in fact the problem seems quite challenging even on manifolds with non-negative sectional curvature, recently some progresses were made on \eqref{bernstein_2}, \eqref{bernstein_3}. Below, we comment on the following recent result, independently due to the authors and to Ding \cite{ding}: 
\begin{theorem}[\cite{cmmr,ding}]\label{teo_B2}
	A complete manifold with $\Ricc \ge 0$ satisfies \eqref{bernstein_2}: positive solutions of \eqref{eq_minimal_Rm} on the entire $M$ are constant.
\end{theorem}
Under the further condition $\Sect \ge - \kappa^2$ for some constant $\kappa>0$, Theorem \ref{teo_B2} was proved by Rosenberg, Schulze and Spruck in \cite{rosenbergschulzespruck}. If the Ricci and sectional curvatures are allowed to be negative somewhere, then the situation is different: in this respect, we quote the recent  \cite{chh_nonexistence} by Casteras, Heinonen and Holopainen, ensuring the constancy of positive minimal graphs with at most linear growth on manifolds with asymptotically non-negative sectional curvature and only one end. Regarding property \eqref{bernstein_3}, the rigidity of manifolds supporting linearly growing, nonconstant minimal graphs was studied in \cite{dingjostxin,dingjostxin2} on manifolds satisfying  
\[
\left\{ \begin{array}{l}
\Ricc \ge 0, \qquad \disp \lim_{r \ra \infty} \frac{|B_r|}{r^m} > 0 \\[0.3cm]
\text{the curvature tensor decays quadratically}.
\end{array}\right.
\]
Together with equation \eqref{eq_minimal_Rm}, that we will investigate in Section \ref{sec_gradient}, in the present survey we will also consider energies containing a potential term, of the type
\[
\Omega \Subset M \qquad \longmapsto \qquad \int_\Omega \Big( \sqrt{1+|Du|^2} + b(x)F(u) \Big) \di x, 
\]
with $0 < b \in C(M)$ and $f \in C^1(\R)$. The corresponding Euler-Lagrange equation becomes
\begin{equation}\label{eq_H}
\diver \left( \frac{Du}{\sqrt{1+|Du|^2}} \right) = b(x)f(u), 
\end{equation} 	
with $f = F'$. This family of equations includes the capillarity equation $b(x) f(u) = \kappa u$, for some constant $\kappa > 0$, that has been thoroughly studied, for instance, in \cite{finn_book}. 

\begin{notation}
	Hereafter, given two positive functions $f,g$, with $f \gtrsim g$ we mean $f \ge C g$ for some positive constant $C$, and similarly for $f \lesssim g$. We write $f \in L^1(+\infty)$ to denote
	\[
	\int^{+\infty} f(s)\di s < +\infty,
	\]
	and similarly for $f \in L^1(0^+)$. When there is no risk of confusion, we write $\infty$ instead of $+\infty,-\infty$. 
\end{notation}

If $f$ is non-decreasing, the behaviour of solutions of \eqref{eq_H} on $\R^m$ is well understood, thanks to the works of Tkachev \cite{tkachev, tkachev_2}, Naito and Usami \cite{nu} and Serrin \cite{Serrin_4}, see also \cite[Thm. 10.4]{farina_minimal}.

\begin{theorem}[\cite{tkachev,nu,Serrin_4}]\label{teo_tkachev}
	If $f$ be continuous and non-decreasing on $\R$, $f \not \equiv 0$ and suppose  
	\[
	b(x) \gtrsim \big( 1 + |x| \big)^{-\mu} \qquad \text{on } \, \R^m,
	\]	
	for some constant $\mu < 1$. Then a solution $u$ of \eqref{eq_H} on $\R^m$ satisfies $f(u) \equiv 0$ on $\R^m$, and thus it is constant by \eqref{bernstein_2}.
\end{theorem}

\begin{remark}
	\emph{Condition $f \not \equiv 0$, together with $f(u) \equiv 0$, guarantees that $u$ is one sided bounded and thus the applicability of \eqref{bernstein_2}. It is therefore essential, as the example of affine functions show. 
	}
\end{remark}

\begin{remark}
	\emph{The result establishes that the only graphs with constant mean curvature (CMC) defined on the entire $\R^m$ must be minimal, a result first proved by Heinz \cite{heinz} for $m=2$, and later extended independently  by Chern \cite{chern} and Flanders \cite{flanders} to all $m\ge 3$. As a matter of fact, Salavessa in \cite{salavessa} observed that the result is true on each manifold with zero Cheeger constant, in particular, on complete manifolds with subexponential volume growth.
	}
\end{remark}

In particular, in the assumptions of Theorem \ref{teo_tkachev} the only solution of the capillarity equation on $\R^m$ is $u \equiv 0$. In Theorem \ref{teo_B2} below, we shall prove that Theorem \ref{teo_tkachev} holds on each manifold with $\Ricc \ge 0$; precisely, the implication $f(u) \equiv 0$ only requires that $M$ has polynomial volume growth of geodesic balls, a condition implied (via Bishop-Gromov comparison), but not equivalent, to $\Ricc \ge 0$. On the contrary, as we shall see below, Theorem \ref{teo_tkachev} fails on manifolds with faster than polynomial volume growth, cf. Remark \ref{rem_polysharp}. We note that the restriction $\mu< 1$ can be relaxed for the capillarity case, the sharp threshold being $\mu < 2$ by recent work of Farina and Serrin \cite{farinaserrin1}, and we will come back on this later, see Theorem \ref{teo_capillarity_pre} below.\par
It is instructive to compare Theorem \ref{teo_tkachev} to a corresponding result for the Laplace operator, say for $b \equiv 1$: To make the statement easier, we assume that $f$ has only one zero, so up to translation $f(0)=0$, $t f(t) > 0$ for $t \in \R \backslash \{0\}$, and set 
\begin{equation}\label{def_F}
F(s) = \int_0^s f(t) \di t.
\end{equation}
Then, by works of Keller \cite{keller} and Osserman \cite{osserman}, the only solution of 
\begin{equation}\label{eq_lapla}
\Delta u = f(u) \qquad \text{on } \, \R^m
\end{equation}
is $u \equiv 0$ if and only if the Keller-Osserman conditions
\begin{eqnarray}
& \disp 	\int^{+\infty} \frac{\di s}{\sqrt{F(s)}} < \infty, & \label{KO} \\
& \disp 	\int_{-\infty} \frac{\di s}{\sqrt{F(s)}} < \infty & \label{KO_menoinfty}
\end{eqnarray}
hold. More precisely, \eqref{KO} guarantees the following Liouville property:
\begin{equation}\label{Lio}
\text{any solution of $\Delta u \ge f(u)$ is bounded above and satisfies $f(u^*) \le 0$,} 
\end{equation}
where $u^* : = \sup_{\R^m} u$. Thus, $u \le 0$ on $\R^m$. The reverse inequality $u \ge 0$ follows by noticing that $\bar u : = -u$ solves $\Delta \bar u = \bar f(\bar u)$ with $\bar f(t) = -f(-t)$, so condition \eqref{KO_menoinfty} is simply a rewriting of \eqref{KO} applied to $\bar f$.

\begin{remark}
	\emph{Regarding the case of non-constant weights $b$, let $f>0$ on $\R^+$, and assume that $f$ is increasing for large enough $t$. If
		\[
		b(x) \gtrsim \big( 1 + |x|\big)^{-2} \qquad \text{on } \, \R^m,
		\]
		then \eqref{KO} is necessary and sufficient to guarantee that solutions of 
		\[
		\Delta u \ge b(x) f(u) \qquad \text{on } \, \R^m
		\]
		satisfy $\sup_M u \le 0$. This result can be obtained, for instance, by combining Theorems 10.3 and Corollary 10.23 in \cite{bmpr}. The quadratic lower bound for $b$ is sharp, since positive entire solutions were constructed in \cite{dughgowa,bmr_JFA} when \eqref{KO} holds and  
		\[
		b(x) \le B(|x|), \qquad \text{with } \, \int^{+\infty} s B(s)\di s < \infty.
		\]
	}
\end{remark}

\begin{remark}
	\emph{We mention that \eqref{KO} is sufficient for the validity of the Liouville property \eqref{Lio} on complete manifolds satisfying 
		\begin{equation}\label{eq_quadraricci} 
		\Ricc(\nabla r, \nabla r) \ge - \kappa^2(1+r^2) \qquad \text{on } \, M \backslash \cut(o),
		\end{equation}
		for some constant $\kappa>0$, where $r$ is the distance from a fixed origin $o$, and $\cut(o)$ is the cut-locus of $o$. This result follows, for example, as a particular case of \cite[Cor 10.23]{bmpr}, and condition \eqref{eq_quadraricci} is sharp for its validity. Furthermore, \eqref{KO} is also a necessary condition for \eqref{Lio} for a large class of manifolds, including each Cartan-Hadamard space, see \cite[Thm. 2.31]{bmpr}. By Bishop-Gromov comparison theorem, note that \eqref{eq_quadraricci} implies the following quadratic exponential bound for the volume of balls $B_r$ centered at $o$:
		\begin{equation}\label{stoccomp}
		\limsup_{r \ra \infty} \frac{\log|B_r|}{r^2} < \infty. 
		\end{equation}
	}
\end{remark}


The absence of a growth requirement on $f$ in Theorem \ref{teo_tkachev} might be partially justified by considering the Laplacian and the mean curvature operator as members of the the family of quasilinear operators
\[
\Delta_\varphi u : = \diver \left( \frac{\varphi(|Du|)}{|Du|} Du \right), 
\]
for $\varphi \in C(\R^+_0) \cap C^1(\R^+)$ with $\varphi'>0$ on $\R^+$, with respectively the choices $\varphi(t) = t$ and $\varphi(t) = t/\sqrt{1+t^2}$. Letting 
\[
K(t) = \int_0^t s \varphi'(s) \di s
\]
denote the pre-Legendre transform of $\varphi$, note that $K$ realizes a diffeomorphism from $\R^+_0$ to $[0, K_\infty)$. If $K_\infty = \infty$, then the integrability qualifying as the right Keller-Osserman condition associated to $f$ and $\Delta_\varphi$ writes as
\begin{equation}\label{KO_general}
\int^{\infty} \frac{\di s}{K^{-1}(F(s))} < \infty,
\end{equation}	
see the detailed account in \cite[Sec. 10]{bmpr}. Note that \eqref{KO_general} rewrites as \eqref{KO} for the Laplace-Beltrami operator. By \cite{nu}, in the assumptions of Theorem \ref{teo_tkachev} on $f$, \eqref{KO_general} is necessary and sufficient for the vanishing of a solution of $\Delta_\varphi u = f(u)$, at least in the case $\varphi(\infty) = \infty$ (that implies $K_\infty = \infty$, cf. also \cite[Remark 10.2]{bmpr}).\par
However, the mean curvature operator does not satisfy $K_\infty = \infty$, and Theorem \ref{teo_tkachev} might suggest that a Liouville theorem for solutions of \eqref{eq_H}, in the spirit of \eqref{Lio}, could hold on reasonably well-behaved manifolds $M$, arguably those satisfying \eqref{eq_quadraricci} or \eqref{stoccomp}, without the need to control the growth of $f$ at infinity. However, we shall see below that a suitably defined Keller-Osserman condition will appear again as soon as $M$ has a volume growth faster than \emph{polynomial}. The guiding example is again $M = \HH^m$, where the existence problem for \eqref{eq_H} was recently studied in \cite{bckrt}.

\subsection{Further ambient spaces, and MCF solitons}

Ambient manifolds of the type $\R \times M$ leave aside various cases of interest, notably interesting representations of the hyperbolic space $\HH^{m+1}$. This and other examples motivated us to consider graphs in more general ambient manifolds $\bar M$, that foliate topologically as products $\R \times M$ along the flow lines of relevant, in general not parallel, vector fields. A sufficiently large class for our purposes is that of warped products
\begin{equation}\label{warpedproducts}
\bar M = \R \times_h M, \qquad \text{with metric} \qquad \bar g = \di s^2 + h(s)^2 g,
\end{equation}
for some positive $h \in C^\infty(\R)$. Note that $X = h(s) \partial_s$ is a conformal field with geodesic flow lines. Given $v : M \ra \R$, we define the graph
$$
\Sigma^m = \Big\{ (s,x) \in \R \times M, \ \ s = v(x)\Big\},
$$
that we call a geodesic graph. For instance, $\HH^{m+1}$ admits the following representations:
\begin{itemize}
	\item[(i)] $\HH^{m+1} = \R \times_{e^s} \R^m$, with the slices of constant $s$, called horospheres, being flat Euclidean spaces with normalized mean curvature $1$ in the direction of $-\partial_s$. This follows by changing variables $s = - \log x_0$ in the upper half-space model
	\begin{equation}\label{eq_uhs}
	\HH^{m+1} = \Big\{ (x_0,x) \in \R \times \R^{m} \ \ : \ \ x_{0}>0\Big\}, \qquad \bar g = \frac{1}{x_0^2}\Big( \di x_0^2 + g_{\R^m}\Big);
	\end{equation}
	\item[(ii)] $\HH^{m+1} = \R \times_{\cosh s} \HH^m$, with the slices $\{s = s_0\}$ for constant $s_0$, called hyperspheres, being totally umbilical hyperbolic spaces of normalized mean curvature $H = \tanh(s_0)$ in the direction of $-\partial_s$, and sectional curvature $- (\cosh s_0)^{-2}$.
\end{itemize}

In cases (i) and (ii) we say that $\Sigma$ is, respectively, a geodesic graph over horospheres and hyperspheres. Do Carmo and Lawson in \cite{docarmolawson} investigated geodesic graphs in $\HH^{m+1}$ with constant mean curvature $H \in [1,1]$, both over horospheres and hyperspheres, and proved the following Bernstein type theorem:
\[
\begin{array}{l}
\text{if $\Sigma$ is over a horosphere, then $\Sigma$ is a horosphere (in particular, $H = \pm 1$)}\\
\text{if $\Sigma$ is over a hypersphere, then $\Sigma$ is a hypersphere.}
\end{array}
\]
Their result is part of a more general statement, that draws the rigidity of a properly embedded CMC hypersurface $\Sigma$ from the rigidity of its trace on the boundary at infinity $\partial_\infty \HH^{m+1}$. The proof, relying on the moving plane method, seems difficult to adapt to variable mean curvature, and the search for an alternative approach was one among the geometric motivations that lead us to write the monograph \cite{bmpr}.\par
The equations for geodesic graphs with prescribed mean curvature $H$ is more transparent when written in terms of the parameter $t$ of the flow $\Phi$ of $X$, related to $s$ by the equation
\begin{equation}\label{bonitinho}
t = \int_0^s \frac{\di \sigma}{h(\sigma)}, \qquad t : \R \ra t(\R) : = I.
\end{equation}
Set  
\begin{equation}\label{eq_uv}
\lambda(t) = h\big( s(t)\big), \qquad u(x) = t\big( v(x)\big),  
\end{equation}
so $u : M \ra I$ and $\Sigma$ writes as $\{ t = u(x)\}$. Clearly, $u = v$ when $\bar M = \R \times M$ with the Riemannian product metric. Computing the normalized mean curvature $H$ of $\Sigma^m$ with respect to the upward--pointing unit normal
\begin{equation}\label{normal_geodesic}
\nu = \frac{1}{\lambda(u)\sqrt{1+|D u|^2}} \Big( \partial_t - (\Phi_u)_* D u \Big),
\end{equation}
we obtain the following equation for $u : M \ra I$:
\begin{equation}\label{prescribed_geodesic}
\diver \left( \frac{D u}{\sqrt{1+|D u|^2}} \right) = m \lambda(u) H + m \frac{\lambda_t(u)}{\lambda(u)} \frac{1}{\sqrt{1+|D u|^2}} \quad \text{on } \, M,
\end{equation}
where $\lambda_t$ is the derivative of $\lambda$ with respect to $t$.
Clearly, if $H = 0$ this is an equation of type
\begin{equation}\label{problems}
\diver\left( \frac{D u}{\sqrt{1+|D u|^2}} \right) = b(x)f(u)\ell(|D u|)\quad \text{on } \, M,
\end{equation}
for suitable $b,f,\ell$. Equation \eqref{problems} will be our focus for the first part of this survey. Another prescription of $H$ that leads to an equation of the same type arises from the study of mean curvature flow solitons. Briefly, consider a mean curvature flow (MCF) evolving from $\Sigma$, that is, a solution $\varphi : [0,T] \times M \ra \bar M$ of
\[
\left\{ \begin{array}{l}
\disp \frac{\partial \varphi_t}{\partial t} = m H_{t} \qquad \text{on } \, M, \\[0.2cm]
\varphi(0,x) = (u(x),x)
\end{array}\right.
\]
with $H_t$ the mean curvature of $\Sigma_t = \varphi_t(M)$ in the direction of the normal $\nu_t$ chosen to match \eqref{normal_geodesic} at time $t=0$; we call $\varphi$ a soliton with respect to a field $Y$ on $\bar M$, if $\varphi$ moves $\Sigma$ along the flow $\Psi$ of $Y$, that is, if there exists a tangent vector field $T$ on $M$ with flow $\eta_t$, and a time reparametrization $t \mapsto \tau(t)$, such that
\[
\varphi(x,t) = \Psi \big( \tau(t), \eta(t,x) \big).
\]
Differentiating at $t=0$, we see that a MCF soliton satisfies the identity $\bar g(Y,\nu) = m H$. Interesting solitons arise when $Y$ is a conformal or Killing field, and in our case the choice $Y = \pm h(s) \partial_s$ allows to rewrite  \eqref{prescribed_geodesic} as the following equation:
\begin{equation}\label{eq_soliton}
\diver\left( \frac{D u}{\sqrt{1+|D u|^2}}\right) = \left[ \frac{m \lambda_t(u) \pm \lambda^3(u)}{\lambda(u)} \right] \frac{1}{\sqrt{1+|D u|^2}} = \frac{m h'(v) \pm h^2(v)}{\sqrt{1+|D u|^2}}
\end{equation}
with $v = t^{-1}(u)$. In particular, a self-translator in $\R^{m+1}$, that is, a soliton for the vertical direction $\partial_s$, satisfies 
\begin{equation}\label{translator}
\diver\left( \frac{D u}{\sqrt{1+|D u|^2}}\right) = \frac{ \pm 1}{\sqrt{1+|D u|^2}}.
\end{equation}
We mention that the existence problem for equations tightly related to \eqref{eq_soliton}, describing certain $f$-minimal graphs in $\R \times M$, has recently been considered in \cite{chh_fminimal}.

\begin{remark}
	\emph{Another relevant class of ambient spaces is that of products $M \times_h \R$ with metric $\bar g =  g + h(x)^2\di s^2$, for some $0 < h \in C^\infty(M)$. In this case $X= \partial_s$ is Killing but not parallel, unless $h$ is constant. For instance, $\HH^{m+1}$ admits two such warped product decompositions, according to whether $X$ has one or two fixed points at infinity: the first can be obtained by isolating the coordinate $s = x_m$ in the upper half-space model \eqref{eq_uhs}, leading to
		\begin{equation}\label{wp_trasla}
		\HH^{m+1} = \HH^m \times_{h} \R \qquad \text{with } \, h(x_0,\ldots, x_{m-1}) = \frac{1}{x_0};
		\end{equation}
		the second can be written as
		\begin{equation}\label{wp_radial}
		\HH^{m+1} = \HH^m \times_{\cosh r} \R, \qquad  \bar g = g_{\HH^m}
		+ \big(\cosh^ 2r(x)\big)\di s^2,
		\end{equation}
		with $r : \HH^m \ra \R$ the distance from a fixed origin in $\HH^m$, and corresponds, in the upper half--space model, to the fibration of $\HH^{m+1}$ via Euclidean lines orthogonal to the totally geodesic hypersphere $\{ x_0^2 + |x|^2 = 1\}$. There is no rigidity for graphs in $\HH^{m+1}$ along these decompositions: the Plateau's problem at infinity for graphs with constant $H \in (0,1)$ is always solvable, by work of Guan and 
		Spruck \cite{guanspruck} for \eqref{wp_radial}, and Ripoll and Telichevesky \cite{rite} for \eqref{wp_trasla}. Existence for the prescribed mean curvature equation on more general products $M \times_h \R$ is studied in \cite{warped}.
	}
\end{remark}

\section{Analytic behaviour of the mean curvature operator}

Hereafter, all solutions of differential equations and inequalities will be meant in the weak sense. In our study of the rigidity problem for equations of the type
\begin{equation}\label{eq_eq_general}
\diver \left( \frac{Du}{\sqrt{1+|Du|^2}} \right) = \mathscr F(x,u,D u)
\end{equation}
on a complete, connected manifold $M$, we assume
\begin{equation}\label{assu_Fgen}
\mathscr F(x,t,p) \ge \frac{f(t)}{(1+r(x))^\mu} \frac{|p|^{1-\chi}}{\sqrt{1+|p|^2}},
\end{equation}
for suitable $\chi \in [0,1]$ and $\mu \in \R$, reducing the problem to the investigation of solutions of
\begin{equation}\label{eq_nostraineq}
\diver \left( \frac{Du}{\sqrt{1+|Du|^2}} \right) \ge \frac{f(u)}{(1+r)^\mu} \frac{|Du|^{1-\chi}}{\sqrt{1+|Du|^2}} \qquad \text{on } \, M.
\end{equation}
Note that, for $\chi \in (0,1)$, the RHS is allowed to vanish both as $|Du| \ra 0$ and as $|Du| \ra \infty$. The form \eqref{assu_Fgen} is sufficiently general to encompass all of our cases of interest: for instance, \eqref{prescribed_geodesic} for $H=0$ is included by setting
\[
\mu = 0, \qquad f(t) = m \frac{\lambda_t(t)}{\lambda(t)} = m h'(s(t)), \qquad \chi =1,
\]
while the capillarity equation
implies the next inequality for each $\chi \in [0,1]$:
\[
\diver \left( \frac{Du}{\sqrt{1+|Du|^2}} \right) = \kappa u \ge \frac{\kappa}{C_\chi} u \frac{|Du|^{1-\chi}}{\sqrt{1+|Du|^2}} \qquad \text{on } \, \{u>0\},
\]
where we used the fact that $t^{1-\chi}/\sqrt{1+t^2} \le C_\chi$ on $\R^+$, for some constant $C_\chi > 0$. By the Pasting Lemma (Kato inequality, cf. \cite{dambrosiomitidieri_2}) $u_+ = \max\{u,0\}$ solves, in the weak sense,
\begin{equation}\label{eq_capi}
\diver \left( \frac{Du_+}{\sqrt{1+|Du_+|^2}} \right) \ge \frac{\kappa}{C_\chi} u_+ \frac{|Du_+|^{1-\chi}}{\sqrt{1+|Du_+|^2}} \qquad \text{on } \, M.
\end{equation}
Thus, we can choose
\[
\mu = 0, \qquad f(t) = \frac{\kappa}{C_\chi} \ \ \text{constant, and any } \, \chi \in [0,1].
\]	
As we shall see, the arbitrariness in the choice of $\chi$ will allow to reach the optimal thresholds for the rigidity of solutions to \eqref{eq_capi}.\par
Our goal in this section is to identify sharp conditions that force any non-constant solution of \eqref{eq_eq_general}, possibly already matching some a-priori bound, to satisfy $f(u) \equiv 0$ on $M$. If $f$ has just isolated zeroes, we would directly deduce that $u$ must be constant, otherwise the rigidity issue for solutions of
\[
\diver \left( \frac{Du}{\sqrt{1+|Du|^2}} \right) = 0
\]
requires further tools and assumptions, described in the next section. The argument to deduce $f(u) \equiv 0$ depends on the fact that any non-constant solution of \eqref{eq_nostraineq}, for suitable $f$ or provided that $u$ grows slowly enough, is automatically bounded from above and satisfies
\[
f(u^*) \le 0,
\]
where, hereafter in this survey,
\[
u^* : = \sup_M u, \qquad u_* : = \inf_M u.
\]		
\begin{remark}[\textbf{Constant solutions}]
	\emph{We observe that a constant function solves \eqref{eq_nostraineq} if and only if either
		\[
		\begin{array}{ll}
		\chi \in [0,1), & \quad \text{independently of } \, f, \ \text{ or} \\[0.2cm]
		\chi =1, & \quad f(u^*) \le 0.
		\end{array}
		\]
		Therefore, in what follows we concentrate on \emph{non-constant} solutions.
	}
\end{remark}	
The conclusion $f(u^*) \le 0$ should be interpreted as a maximum principle at infinity, in the spirit of \cite{prsmemoirs}. Indeed, if $M$ is compact and we assume $u \in C^2(M)$ and $\chi =1$, $f(u^*)\le 0$ follows readily by evaluating \eqref{eq_nostraineq} at a point $x_0$ realizing $u^*$, and using that $|Du(x_0)|=0$, $D^2 u(x_0) \le 0$. We shall be more detailed in the next section. To constrain solutions of \eqref{eq_eq_general}, the idea is to apply the maximum principle at infinity both to $u$ and to $\bar u : = -u$, which is possible if, for instance, $\mathscr{F}$ satisfies conditions like 
\[
\mathscr{F}(x,t,p) = -\mathscr{F}(x,-t,-p). 
\]
In this case, $\bar u$ solves
\[
\diver \left( \frac{D \bar u}{\sqrt{1+|D \bar u|^2}} \right) \ge \frac{\bar f(\bar u)}{(1+r)^\mu} \frac{|D\bar u|^{1-\chi}}{\sqrt{1+|D \bar u|^2}}, \qquad \text{with } \, \bar f(t) := -f(-t),
\]
and $\bar f(\bar u^*) \le 0$ rewrites as $f(u_*) \ge 0$. If, for instance, $f$ is increasing, we would eventually conclude the identity $f(u) \equiv 0$ on $M$. \par
Taking into account these facts, in the present section we focus on the maximum principle at infinity for solutions of inequality \eqref{eq_nostraineq}, and on related Liouville properties. Our purpose is to guide the reader through the results in \cite{bmpr}, discuss the sharpness of the assumptions and their interplay, and illustrate some geometric applications. For reasons of room, no proofs of the analytic results will be given or sketched, but their applications to geometric problems will be commented in some details.

\subsection{Weak and strong maximum principles at infinity}\label{subsec_WSMP}

For $u$ bounded from above, the necessity to deduce $f(u^*)\le 0$ even if $u^*$ is not attained motivated the introduction of the fundamental Omori-Yau maximum principles, \cite{omori,yau,chengyau}, that generated an active area of research related to potential theory on manifolds \cite{maripessoa,maripessoa_2}, with applications to a variety of geometric problems. We refer the reader to  \cite{AMR_book,bmpr,prsmemoirs} for a detailed account and an extensive set of references. In its original formulation, the Omori-Yau principle for the Laplace-Beltrami operator is the property that, for every $u \in C^2(M)$ bounded from above, there exists a sequence $\{x_j\} \subset M$ satisfying
\[
u(x_j) \ra u^*, \qquad |D u(x_j)| < \frac{1}{j}, \qquad \Delta u(x_j) < \frac{1}{j}.
\]
Clearly, if $u^*$ is attained at some $x$ then the sequence can be chosen to be constantly equal to $x$. The Omori-Yau principle is also called, in \cite{AMR_book,bmpr}, the \emph{strong maximum principle at infinity}. The fact that the gradient condition $|Du(x_j)|< 1/j$ is unnecessary in many applications motivated the introduction, in \cite{prsmemoirs}, of the related \emph{weak maximum principle at infinity}, asking for the validity, along a suitable sequence $\{x_j\}$, just of bounds   
\[
u(x_j) \ra u^*, \qquad \Delta u(x_j) < \frac{1}{j}.
\]
Examples show that the two are indeed different, that is, there exist manifolds satisfying the weak but not the strong maximum principle, cf. \cite{borbely,bmpr}. The weak principle turns out to have important links with the theory of stochastic processes, explored in \cite{prsmemoirs} and recalled later in this paper in more detail. In our setting, we shall look at suitable formulations of weak and strong principles, adapted to solutions with low regularity and to mean curvature type operators. Extension to more general operators can be found in \cite{bmpr} (quasilinear ones) and in \cite{maripessoa,maripessoa_2} (fully nonlinear ones).\par
Hereafter, for $\gamma \in \R$ and $\eps \in \R^+$ we use the shorthand notations $\{u> \gamma\}$ and $\{ u> \gamma, \ |Du|< \eps\}$ to mean, respectively, the sets
\[
\big\{ x \ : \ u(x) > \gamma\big\}, \qquad   \big\{ x \ : \ u(x) > \gamma, \ |Du(x)|< \eps \big\}
\]
\begin{definition}[\cite{bmpr}]
	Given $\mu \in \R$, $\chi \in [0,1]$, we say that
	\begin{itemize}
		\item the \emph{\textbf{strong maximum principle at infinity}}, $\smp$, holds if whenever $u \in C^1(M)$ is non-constant, bounded from above and solves
		\begin{equation}\label{eq_SMP_vera}
		\diver \left( \frac{Du}{\sqrt{1+|Du|^2}}\right) \ge \frac{K}{(1+r)^\mu} \frac{|Du|^{1-\chi}}{\sqrt{1+|Du|^2}} \qquad \text{on } \, \big\{ u > \gamma, \ |Du| < \eps \big\}
		\end{equation}
		for some constant $\eps>0$ and $K,\gamma \in \R$, then $K \le 0$.
		\item the \emph{\textbf{weak maximum principle at infinity}}, $\wmp$, holds if whenever $u \in \lip_\loc(M)$ is non-constant, bounded from above and solves
		\begin{equation}\label{eq_WMP_vera}
		\diver \left( \frac{Du}{\sqrt{1+|Du|^2}}\right) \ge \frac{K}{(1+r)^\mu} \frac{|Du|^{1-\chi}}{\sqrt{1+|Du|^2}} \qquad \text{on } \, \big\{ u > \gamma \big\}
		\end{equation}
		for some constant $K,\gamma \in \R$, then $K \le 0$.
	\end{itemize}
\end{definition}

\begin{remark}
	\emph{If $M$ is complete, by Ekeland's principle the set $\{ u > \gamma, \ |Du| < \eps\}$ is non-empty for each choice of $\gamma < u^*$ and $\eps>0$. The restriction $u \in C^1(M)$ is made necessary for such set to be open. Suitable notions of $\smp$ for viscosity solutions that are less regular than $C^1$ have been studied in \cite{maripessoa,maripessoa_2} in the setting of fully nonlinear potential theory.
	}
\end{remark}

\begin{remark}
	\emph{For $\chi = 1$ and $u \in C^2(M)$, then $\smp$ can equivalently be formulated as the existence of a sequence $\{x_j\} \subset M$ such that
		\begin{equation}\label{prop_wmp}
		u(x_j) \ra u^*, \qquad |Du(x_j)| < \frac{1}{j}, \qquad  \diver \left( \frac{Du}{\sqrt{1+|Du|^2}}\right)(x_j) < \frac{1}{j},
		\end{equation}
		making contact with the original formulation of Omori and Yau.
	}
\end{remark}

The definition of $\wmp$ is localized on upper level sets $\{u> \gamma\}$ of $u$, for $\gamma \in \R$. However, with the aid of the Pasting Lemma, it can easily be seen that the following statements are equivalent (cf. \cite[Prop. 7.4]{bmpr}):
\begin{itemize}
	\item[(i)] $\wmp$ holds;
	\item[(ii)] for every $f \in C(\R)$, any non-constant solution of 
	\[
	\diver \left( \frac{Du}{\sqrt{1+|Du|^2}} \right) \ge \frac{f(u)}{(1+r)^\mu} \frac{|Du|^{1-\chi}}{\sqrt{1+|Du|^2}} \qquad \text{on } \, M
	\]
	that is bounded from above satisfies $f(u^*) \le 0$.
\end{itemize}

The applications described below will motivate the study of the above maximum principles at infinity. In particular, we mention that although the $\wmp$ can be successfully used to study the rigidity of entire minimal graphs and MCF solitons, to obtain sharp results for hypersurfaces satisfying \eqref{prescribed_geodesic} with more general choices of $H$ we had to appeal to the full strentgh of the $\smp$. 

\begin{remark}
\emph{The above definition suggests that $\wmp$ can be interpreted as a comparison principle with constant functions on open, possibly unbounded sets, namely the upper level sets of $u$. We mention that the development of a comparison theory for solutions of \eqref{eq_eq_general} on unbounded sets is a rather interesting issue for which many questions are still open, even if the right-hand side does not depend on $Du$. Some sharp results in this respect can be found in \cite{prspac}, see also the references therein.}
\end{remark}

Examples show that some restriction on the geometry of $M$ has to be required in order to deduce the validity of $\wmp$ and $\smp$. Concerning the weak maximum principle, we prove in \cite[Thm. 7.5]{bmpr} the following sharp

\begin{theorem}[Thm. 7.5 in \cite{bmpr}]\label{teo_main_2_wmp}
	Let $M$ be a complete Riemannian manifold, and let
	\[
	\chi \in [0,1], \qquad \mu \le \chi+1.
	\]
	If
	\begin{equation}\label{volgrowth_sigmazero}
	\begin{array}{lll}
	\mu < \chi+1 & \text{and} & \disp \qquad \liminf_{r \ra \infty} \frac{\log|B_r|}{r^{\chi+1-\mu}} < \infty \quad (=0 \, \text{ if } \, \chi=0); \\[0.4cm]
	\mu = \chi+1 & \text{and} & \disp \qquad \liminf_{r \ra \infty} \frac{\log|B_r|}{\log r} < \infty \quad (\le 2 \, \text{ if } \, \chi=0),
	\end{array}
	\end{equation}
	then $\wmp$ holds. In particular, for $f \in C(\R)$, any non-constant solution $u \in \lip_\loc(M)$ of
	\[
	\diver \left( \frac{Du}{\sqrt{1+|Du|^2}}\right) \ge \frac{f(u)}{(1+r)^\mu} \frac{|Du|^{1-\chi}}{\sqrt{1+|Du|^2}} \qquad \text{on } \, M
	\]
	that is bounded from above satisfies $f(u^*) \le 0$.
\end{theorem}	

Concerning the strong maximum principle at infinity, it is still an open problem to decide whether the volume growth \eqref{volgrowth_sigmazero} is sufficient. In \cite{bmpr}, we obtained sharp conditions by imposing the next lower bound on the Ricci curvature:

\begin{equation}\label{ricciassu_intro}
\Ricc (\nabla r, \nabla r) \ge -(m-1)\kappa^2\big( 1+r^2\big)^{\alpha/2} \qquad \text{on } \, M \backslash \cut(o),
\end{equation}
for constants $\kappa \ge 0, \alpha \ge -2$, where $r$ is the distance from a fixed origin $o$. Namely, we compare $M$ to the radially symmetric model $M_g$ being $\R^m$ endowed, in polar coordinates $(r,\theta)$,  with the metric
\[
\di s^2 = \di r^2 + g(r)^2 \di \theta^2,
\]
where $\di \theta^2$ is the metric on the unit sphere $\mathbb{S}^{m-1}$ and $g\in C^{\infty}(\R^+_0)$ is a solution of
\[
\left\{ \begin{array}{l}
g'' = \kappa^2(1+r^2)^{\alpha/2} g \qquad \text{on } \, \R^+, \\[0.2cm]
g(0)=0, \qquad g'(0)=1.
\end{array}\right.
\]
The case $\kappa=0$ corresponds to choosing the Euclidean space as model $M_g$, described by the function $g(r) = r$, while the case $\alpha=0$ corresponds to choosing the Hyperbolic space of curvature $-\kappa^2$, for which
\[
g(r) = \frac{\sinh(\kappa r)}{\kappa}.
\]
By \cite[Prop. 2.11]{prs} and the Bishop-Gromov comparison theorem, the volume of the geodesic balls $B_r$ and $\BB_r$ in $M$ and $M_g$, respectively, relate as follows:
\begin{equation}\label{crescitevol_counter}
\log|B_r| \le \log|\BB_r| \sim \left\{ \begin{array}{ll} \dfrac{2\kappa(m-1)}{2+\alpha} r^{1+ \alpha/ 2} & \quad \text{if } \, \alpha>-2, \\[0.2cm]
\big[(m-1)\bar\kappa +1\big] \log r & \quad \text{if } \, \alpha = -2,
\end{array}\right.
\end{equation}
as $r \ra \infty$, where $\bar \kappa = (1+\sqrt{1+4\kappa^2})/2$. Therefore, the inequalities
\[
\chi \in (0,1], \qquad \mu \le \chi - \frac{\alpha}{2}
\]
imply any of the conditions \eqref{volgrowth_sigmazero_Linfty_inthetheorem} with $\chi>0$:
\[
\begin{array}{lll}
\mu < \chi+1 & \text{and} & \disp \qquad \liminf_{r \ra \infty} \frac{\log |B_r|}{r^{\chi+1-\mu}} < \infty \qquad \text{or}\\[0.4cm]
\mu = \chi+1 & \text{and} & \disp \qquad \liminf_{r \ra \infty} \frac{\log|B_r|}{\log r} < \infty,
\end{array}
\]
respectively if either $\alpha>-2$ or $\alpha = -2$. Thus, the growth thresholds in the following (sharp) sufficient condition for $\smp$ in \cite[Thm. 8.5]{bmpr} turn out to be the same as for $\wmp$.

%

\begin{theorem}[Thm. 8.5 in \cite{bmpr}]\label{teo_SMP_intro}
	Let $M$ be a complete $m$-dimensional manifold satisfying \eqref{ricciassu_intro}, for some $\kappa \ge 0$, $\alpha \ge -2$. If
	\[
	\chi \in (0,1], \qquad \mu \le \chi - \frac{\alpha}{2}.
	\]
	Then, the $\smp$ holds.
\end{theorem}

\subsection{Application: minimal and prescribed mean curvature graphs}  	

We present the following applications of the weak and strong maximum principles at infinity, that generalize Do Carmo-Lawson's theorem stated in the introduction. We begin with entire minimal graphs, proving

\begin{theorem}[\cite{bmpr}, Thms. 7.17, 7.18]\label{teo_bern_minimal_intro}
	Let $M$ be a complete manifold, and consider the warped product $\bar M = \R \times_h M$, with warping function $h$ satisfying either
	\begin{itemize}
		\item[$(i)$] $h$ is convex and $h^{-1} \in L^1(-\infty) \cap L^1(+\infty)$, or
		\item[$(ii)$] $h'>0$ on $\R$, $h'(s) \ge C$ for $s >>1$ and $h^{-1} \in L^1(+\infty)$.
	\end{itemize}
	If
	$$
	\liminf_{r \ra \infty} \frac{\log|B_r|}{r^2} < \infty,
	$$
	then
	\begin{itemize}
		\item[] under $(i)$, every entire minimal graph $v : M \ra \R$ over $M$ is bounded and satisfies $h'(v) \equiv 0$ on $M$. In particular, $v$ is constant if $h$ is strictly convex.
		\item[] under $(ii)$, there exists no entire minimal graph over $M$.
	\end{itemize}
\end{theorem}

\begin{remark}
	\emph{The integrability conditions in $(i),(ii)$ mean, loosely speaking, that the warped product $\bar M$ opens up faster than linearly for large $s$. The case of cones $h(s) = s$ is thus borderline. To recover the results in \cite{docarmolawson}, the choice $M = \HH^m$ and $h(s) = \cosh s$ leads to the constancy of entire minimal graphs over hyperspheres in $\HH^{m+1}$, while the choice $M = \R^m$ and $h(s) = e^s$ allows us to recover the non-existence of entire minimal graphs over horospheres.
	}
\end{remark}

\begin{proof}[Proof: sketch]
	In view of \eqref{prescribed_geodesic} and the minimality of $\Sigma$, the function $u=t(v)$, $u : M \ra I$, solves
	\begin{equation}\label{eq_docarmolawson}
	\diver \left( \frac{D u}{\sqrt{1+|D u|^2}}\right) = \frac{f(u)}{\sqrt{1+|D u|^2}},
	\end{equation}
	where
	\[
	f(u) = m \frac{\lambda_t(u)}{\lambda(u)} = m h'(v),
	\]
	while $u$ and $v$ are related by \eqref{eq_uv}. Under assumption $(i)$, $I = (t_*,t^*)$ is a bounded interval and $f(t) \ge C$ for $t$ outside of a compact set of $I$, with $C>0$ a constant. The conclusion $m^{-1} h'(v^*) = f(u^*)\le 0$ follows by applying Theorem \ref{teo_main_2_wmp} (if $u$ is non-constant), with the choice $\mu = 0$, $\chi=1$, or by a direct check if $u$ is constant. Inequality $h'(v_*) \ge 0$ follows along the same lines, and thus $h'(v)\equiv 0$ by the convexity of $h$. 
	
	In case $(ii)$, then $t^* < \infty$ and hence $u$ is bounded from above. Moreover, $f > 0$ on $(t_*,t^*)$, thus \eqref{eq_docarmolawson} admits no constant solutions. The absence of non-constant solutions $u$ follows from $f>0$ by applying again Theorem \ref{teo_main_2_wmp}.
\end{proof}

Concerning graphs with non-constant mean curvature, $\wmp$ seems to be not sufficient to conclude rigidity for the more involved equation \eqref{prescribed_geodesic}, and we use the full strength of $\smp$ to obtain the following Theorem. For simplicity, we state the result for the warping function $h(s) = \cosh s$, but the proof can easily be extended to each $h \in C^2(\R)$ satisfying
\[
\left\{ \begin{array}{l}
h \ \ \text{ even},\\[0.2cm]
h^{-1} \in L^1(-\infty) \cap L^1(+\infty), \\[0.2cm]
(h'/h)' >0 \qquad \text{on } \, \R.
\end{array}\right.
\]

\begin{theorem}[\cite{bmpr}, Thm. 8.11]\label{teo_presc_nonconst}
	Let $M$ be complete and let $\bar M = \R \times_{\cosh s} M$. Assume that the Ricci tensor of $M$ satisfies
	$$
	\Ricc(\nabla r, \nabla r) \ge -(m-1) \kappa^2(1+r)^2 \qquad \text{on } \, M \backslash \cut(o),
	$$
	for some constant $\kappa > 0$. Fix a constant $H_0 \in (-1,1)$, and consider an entire geodesic graph of $u : M \ra \R$ with prescribed curvature $H(x) \ge -H_0$ in the upward direction. Then, $u$ is bounded from above and satisfies
	\begin{equation}\label{bound_graph}
	u^* \le \mathrm{arctanh}(H_0).
	\end{equation}
	In particular,
	\begin{itemize}
		\item[(i)] there is no entire graph with prescribed mean curvature satisfying $|H(x)| \ge 1$ on $M$;
		\item[(ii)] the only entire graph with constant mean curvature $H_0 \in (-1,1)$ in the upward direction is the totally umbilic slice $\{s = \mathrm{arctanh}(H_0)\}$.
	\end{itemize}
\end{theorem}

\begin{remark}
	\emph{In particular, the choice $M = \HH^m$ enables us to recover the statement of Do Carmo-Lawson's result \cite{docarmolawson} for CMC graphs over hyperspheres.
	}
\end{remark}

\begin{proof}[Proof: sketch]
	Defining $t, \lambda(t)$ and $u(x)$ as in \eqref{bonitinho} and \eqref{eq_uv} with $h(s) = \cosh s$, the corresponding  function $u : M \ra \left(-\frac{\pi}{2},\frac{\pi}{2}\right)$ satisfies
	\begin{equation}
	\begin{array}{lcl}
	\disp \diver\left( \frac{D u}{\sqrt{1+|D u|^2}} \right) & = & \disp m \cosh v \left[ H(x) + \frac{\tanh v}{\sqrt{1+|D u|^2}}\right] \\[0.5cm]
	& \ge & \disp m \cosh v \left[- H_0 + \frac{\tanh v}{\sqrt{1+|D u|^2}}\right].
	\end{array}
	\end{equation}
	If, by contradiction, the following upper level set of $v$ (hence, of $u$) is non-empty for some $\gamma\in \R^+$:
	$$
	\Omega_\gamma =\{ \tanh v > H_0 + \gamma\},
	$$
	then, for $\eps>0$ is small enough, on the (non-empty) set
	\[
	\Omega_{\eta,\eps} = \Omega_\gamma \cap \{ |D u|< \eps\}
	\]
	it holds
	\[
	\begin{array}{lcl}
	\disp \disp \diver\left( \frac{D u}{\sqrt{1+|D u|^2}} \right) & \ge & \disp \frac{m \cosh v}{\sqrt{1+|D u|^2}} \left[\gamma - H_0(\sqrt{1+|D u|^2} -1)\right] \\[0.5cm]
	& \ge & \disp \frac{m \gamma}{2\sqrt{1+|D u|^2}}.
	\end{array}
	\]
	Since $u$ is bounded above, applying Theorem \ref{teo_SMP_intro} with $\alpha = 1$, $\mu = 0$, $\chi = 1$ we reach the desired contradiction.\par
	To prove $(i)$, suppose that $|H(x)| \ge 1$ on $M$. Since $\cosh s$ is even, the graph of $-v$ has curvature $-H(x)$ in the upward direction. Thus, up to replacing $v$ with $-v$ we can suppose that $H(x) \ge 1$. Applying the first part of the theorem to any $H_0 > -1$ we obtain $v^* \le \mathrm{arctanh}(H_0)$, and the non-existence of $v$ follows by letting $H_0 \ra -1$.\par
	To prove $(ii)$, let $H(x) = -H_0 \in (-1,1)$ be the mean curvature of the graph of $v$ in the upward direction. Then, Theorem \ref{teo_presc_nonconst} gives $\tanh v^* \le H_0$. On the other hand, the graph of $-v$ has mean curvature $H(x) = H_0$ in the upward direction, and applying again Theorem \ref{teo_presc_nonconst} we deduce $\tanh [(-v)^*] \le -H_0$, that is, $\tanh v_* \ge H_0$. Combining the two estimates gives $v \equiv \mathrm{arctanh}(H_0)$, as required.
\end{proof}

Our results also apply to Schwarzschild, ADS-Schwarzschild and Reissner-Nordstr\"om-Tangherlini spaces, considered in \cite{cmr}, that we now briefly recall. Having fixed a mass parameter $\mathfrak{m}>0$ and a compact Einstein manifold $(M,\sigma)$ with $\Ricc = (m-1)\sigma$, the Schwarzschild space is the product
\begin{equation}\label{schwarz}
\bar{M}^{m+1} = (r_0(\mathfrak{m}), \infty)\times M \qquad \text{with metric} \qquad \bar g = \frac{\di r^2}{V(r)} + r^2 \sigma 
\end{equation}
where  
\begin{equation}\label{def_Vr_schwarz}
V(r) = 1-2 \mathfrak{m} r^{1-m}
\end{equation}
and $r_0(\mathfrak{m})$ is the unique positive root of $V(r)=0$; $\bar M$ generates a Lorentzian manifold $\RR \times \bar{M}$ with metric $-V(r) \di t^2 + \metric_{\overline{M}}$ that solves the Einstein field equation in vacuum with zero cosmological constant. Similarly, given $\bar \kappa \in \{1,0,-1\}$ the ADS-Schwarzschild space (with, respectively, spherical, flat or hyperbolic topology according to whether $\bar \kappa = 1,0,-1$) is the manifold \eqref{schwarz} with respectively
\begin{equation}\label{def_Vr_ADS}
V(r) = \bar \kappa + r^2 - 2 \mathfrak{m} r^{1-m}, \qquad \Ricc = (m-1)\bar \kappa \sigma, 
\end{equation}
and $r_0(\mathfrak{m})$, as before, the unique positive root of $V(r)=0$. They generate (static) solutions of the vacuum Einstein field equation with negative cosmological constant, normalized to be $-m(m+1)/2$. The change of variables
\begin{equation}\label{def_tr_sch}
s = \int_{r_0(\mathfrak{m})}^r \frac{\di \sigma}{\sqrt{V(\sigma)}}, \qquad h(s) = r(s)
\end{equation}
allows to write\footnote{In this respect, note that $r_0(\mathfrak{m})$ is a simple solution of $V(r)=0$, so the integral defining $s$ is convergent, and that $1/\sqrt{V(r)} \not \in L^1(\infty)$ so $s$ is defined on the entire $\R^+$.} $\bar M = \R^+ \times_h M$. Similarly, the Reissner-Nordstr\"om-Tangherlini space is a  charged black-hole solution of Einstein equation, described by \eqref{schwarz} with the choice
$$
V(r) = 1-2 \mathfrak{m} r^{1-m} + \mathfrak{q}^2 r^{2-2m}, \qquad M = \Sph^{m},
$$
with $\mathfrak{q} \in [-\mathfrak{m},\mathfrak{m}]$ being the charge. We consider solitons with respect to the conformal field 
\[
h(s) \partial_s = r \sqrt{V(r)} \partial_r.
\]	
A direct application of Theorem \ref{teo_main_2_wmp} shows the following:

\begin{theorem}[Thm. E in \cite{cmr}]
	There exists no entire graph in the Schwarzschild and ADS-Schwarz\-schild space (with spherical, flat or hyperbolic topology), over a complete $M$, that is a soliton with respect to the field $r \sqrt{V(r)} \partial_r$.
\end{theorem}

\begin{proof}
	If $\Sigma$ is the graph of $\{ s = v(x)\}$ for $v : M \ra \R^+$, by \eqref{eq_soliton} the function $u = t(v)$ satisfies
	\begin{equation}\label{eq_schw}
	\diver \left( \frac{Du}{\sqrt{1+|Du|^2}} \right) = \frac{m h'(v) + h^2(v)}{\sqrt{1+|Du|^2}}.
	\end{equation}
	Differentiating,
	$$
	h'(s) = \frac{\di r}{\di s} = \sqrt{V(r(s))} >0, \qquad h''(s) = \frac 12 \frac{\di V}{\di r}(r(s)) > 0,
	$$
	so that $f(t) = m h'(s(t)) + h^2(s(t))>0$ on $\R^+$. If $M$ is compact, this is sufficient to guarantee that \eqref{eq_schw} does not admit solutions, by the classical maximum principle, and settles the case of Schwarzschild and ADS-Schwarzschild spaces with spherical topology. In the remaining cases, observe that $V(r) \sim r$ for large $r$, thus in view of \eqref{def_tr_sch}
	$$
	h(s) = r(s) \sim e^s \qquad \text{as } \, s \ra \infty.
	$$
	As a consequence, the flow parameter $t(s)$, hence $u$, is bounded from above. The Einstein condition and Bishop-Gromov comparison theorem imply
	$$
	\limsup_{r \ra \infty} \frac{\log |B_r|}{r} < \infty.
	$$
	To conclude, we apply Theorem \ref{teo_main_2_wmp} to deduce $f(u^*) \le 0$  and to reach the required contradiction.
\end{proof}

\subsection{Maximum principles at infinity for solutions with controlled growth}

We next examine the case of solutions with a controlled growth at infinity, to which the second part of \cite[Thm. 7.5]{bmpr} applies: 

\begin{theorem}[Thm. 7.5 in \cite{bmpr}]\label{teo_main_2}
	Let $M$ be a complete Riemannian manifold, fix
	\begin{equation}\label{pararange_2}
	\chi \in [0,1], \qquad \mu \le \chi+1,
	\end{equation}
	and a function $f \in C(\R)$ satisfying
	\[
	f(t) \ge C > 0 \quad \text{for } \, t>>1,
	\]
	for some constant $C>0$. Let $u \in \lip_\loc(M)$ be a non-constant solution of
	\begin{equation}\label{eq_u_WMP}
	\diver \left( \frac{Du}{\sqrt{1+|Du|^2}}\right) \ge \frac{f(u)}{(1+r)^\mu} \frac{|Du|^{1-\chi}}{\sqrt{1+|Du|^2}} \qquad \text{on } \, M
	\end{equation}
	such that
	\begin{equation}\label{opequeno}
	u_+(x) = o\big(r(x)^\sigma\big) \qquad \text{as } \, r(x) \ra \infty,
	\end{equation}
	for some $\sigma>0$ satisfying
	\begin{equation}\label{ipo_sigma_teomain2}
	\chi \sigma \le \chi+1-\mu.
	\end{equation}
	If one of the following properties hold:
	\begin{equation}\label{volgrowth_sigmamagzero}
	\hspace{-0.3cm}\begin{array}{lll}
	(i) & \chi \sigma < \chi+1-\mu, & \disp \liminf_{r \ra \infty} \frac{\log|B_r|}{r^{\chi+1-\mu -\chi\sigma}} < \infty \quad (=0 \, \text{ if } \, \chi=0);\\[0.4cm]
	(ii) & \chi \sigma = \chi+1-\mu, \ \ \ \chi>0, & \disp \liminf_{r \ra \infty} \frac{\log|B_r|}{\log r} < \infty; \\[0.4cm]
	(iii) & \chi=0, \ \mu = 1, &  \disp  \liminf_{r \ra \infty} \frac{\log|B_r|}{\log r} \le \left\{ \begin{array}{ll}
	2-\sigma & \text{if } \, \sigma \le 1, \\[0.1cm]
	\bar p - \sigma(\bar p -1) \\
	\text{for some $\bar p >1$} & \text{if } \, \sigma >1,
	\end{array}\right.
	\end{array}
	\end{equation}
	then, $u$ is bounded above on $M$ and $f(u^*) \le 0$.
\end{theorem}

\begin{remark}[\textbf{Localization to upper level sets}]\label{rem_localization}
	\emph{Theorem \ref{teo_main_2} follows from \cite[Thm.7.15]{bmpr}, where we considered solutions of \eqref{eq_WMP_vera} under the weaker condition
		\[
		\hat u : = \limsup_{r(x) \ra \infty} \frac{u_+(x)}{r(x)^\sigma} < \infty,
		\]
		for some $\sigma \ge 0$. The goal is an upper bound for $K$ of the form
		\begin{equation}\label{eq_cuchi}
		K \le c \hat u^\chi,
		\end{equation}
		for some explicit constant $c\ge 0$ that vanishes in various cases of interest. Note that, for instance, \eqref{eq_cuchi} readily imply $K \le 0$ whenever \eqref{opequeno} and $\chi>0$ hold.
	}
\end{remark}

\begin{remark}[\textbf{The limiting case $\chi = 0$}]\label{rem_chizero}
	\emph{This corresponds to the most ``singular" case, and sometimes it requires an ad-hoc treatment. Conditions \eqref{ipo_sigma_teomain2} and \eqref{volgrowth_sigmamagzero} rephrase for $\chi = 0$ in the following equivalent form:
		\begin{equation}\label{volgrowth_sigmazero_strano}
		\begin{array}{rlll}
		(i) & \quad \mu<1, & \disp \quad \liminf_{r \ra \infty} \frac{\log |B_r|}{r^{1-\mu}} = 0, & \text{ and } \, \sigma>0, \ \text{ or}  \\[0.4cm]
		(iii) & \quad \mu=1, & \disp \quad \liminf_{r \ra \infty} \frac{\log |B_r|}{\log r} = d_0, & \text{ and } \, 0 < \sigma \le \left\{ \begin{array}{ll}
		2-d_0 & \text{if } \, d_0 \ge 1, \\[0.3cm]
		\dfrac{\bar p-d_0}{\bar p-1} & \text{if } \, d_0 < 1,
		\end{array}\right.
		\end{array}
		\end{equation}
		for some $\bar p > 1$. Furthermore, because of \cite[Rem. 7.6]{bmpr}, the conclusion of the theorem holds under the weaker assumption
		\[
		u_+(x) = O\big(r(x)^\sigma\big) \quad \text{as } \, r(x) \ra \infty.
		\]
		We underline that $(i)$ guarantees the absence of solutions with polynomial growth to
		\[
		\diver \left( \frac{Du}{\sqrt{1+|Du|^2}}\right) \gtrsim \frac{1}{(1+r)^\mu} \frac{|Du|}{\sqrt{1+|Du|^2}} \quad \text{on } \, M
		\]
		whenever 	
		\[
		\mu<1, \quad \liminf_{r \ra \infty} \frac{\log|B_r|}{r^{1-\mu}} = 0.
		\]
	}
\end{remark}

\begin{remark}[\textbf{Sharpness}]\label{rem_sharpness}
	\emph{In Subsection 7.4 of \cite{bmpr}, we show that the radially symmetric model $M_g$ described in Subsection \ref{subsec_WSMP} admits an unbounded, radial solution $u \in \lip_\loc(M)$, increasing as a function of the distance $r$ from the origin and solving
		\begin{equation}\label{aim}
		\diver \left( \frac{Du}{\sqrt{1+|Du|^2}} \right) \gtrsim \frac{1}{(1+r)^{\mu}} \frac{|Du|^{1-\chi}}{\sqrt{1+|Du|^2}} \qquad \text{on } \, M,
		\end{equation}
		that coincides with $r^\sigma$ for $r \ge 2$ whenever $\chi \in [0,1]$, $\mu \le \chi+1$ and one of the following conditions hold:
		\begin{equation}\label{ipo_volume_couter}
		\begin{array}{ll}
		1) & \quad \chi\sigma > \chi+1-\mu; \\[0.2cm]
		2) & \disp \quad \chi \sigma = \chi+1-\mu, \qquad \text{and} \quad \lim_{r \ra \infty} \frac{\log |B_r|}{\log r} = \infty; \\[0.4cm]
		3) & \quad \chi \sigma = \chi+1-\mu, \qquad \sigma \in (0,1] \qquad \text{and} \\[0.2cm]
		& \disp \lim_{r \ra \infty} \frac{\log|B_r|}{\log r} = d_0 \in (2 -\sigma, \infty); \\[0.3cm]
		4) & \disp \quad \chi \sigma < \chi+1-\mu, \qquad \text{and} \quad \lim_{r \ra \infty} \frac{\log|B_r|}{r^{\chi+1-\mu-\chi\sigma}} = d_0 \in (0, \infty),
		\end{array}
		\end{equation}
		while it coincides with $r(x)^\sigma/\log r$ for $r \ge 2$ when either
		\begin{equation}\label{ipo_volume_couter_2}
		\begin{array}{ll}
		5) & \quad \chi \sigma< \chi+1-\mu, \qquad \chi>0, \qquad \text{and} \\[0.2cm]
		& \disp \lim_{r \ra \infty} \frac{\log|B_r|}{r^{\chi+1-\mu-\chi\sigma}} = \infty, \qquad \text{or}\\[0.4cm]
		6) & \quad \chi \sigma< \chi+1-\mu, \qquad \chi = 0, \qquad \text{and} \\[0.2cm]
		& \disp \lim_{r \ra \infty} \frac{\log|B_r|}{r^{\chi+1-\mu-\chi\sigma}}  \in (0, \infty).
		\end{array}
		\end{equation}
		The existence of $u$ under any of $1),\ldots, 6)$ above shows the sharpness of the parameter ranges \eqref{pararange_2} and \eqref{ipo_sigma_teomain2}, and of the growth conditions \eqref{opequeno} for $u$ and \eqref{volgrowth_sigmamagzero} for $|B_r|$. In particular,
		\begin{itemize}
			\item[-] in $(2)$, all the assumptions are satisfied but the second or third in \eqref{volgrowth_sigmamagzero}, where the liminf is $\infty$, while in $(3)$, the liminf in the third in \eqref{volgrowth_sigmamagzero} is finite but bigger than the threshold $2-\sigma$ for $\sigma \le 1$;
			\item[-] in $(4)$ the requirements in the first of \eqref{volgrowth_sigmamagzero} are all met, but $u_+ = O(r^\sigma)$ instead of $u_+ = o(r^\sigma)$.
			\item[-] in $(5)$ and $(6)$, $u_+ = o(r^\sigma)$ but the requirements in the first of \eqref{volgrowth_sigmamagzero} barely fail.
		\end{itemize}
	}
\end{remark}

\subsection{Application: solitons for the mean curvature flow}

Our first result considers solitons for the MCF in $\R \times M$ that are realized as the graph of $u : M \ra \R$, and move in the direction of the parallel field $\partial_s$ (called graphical self-translators). In Euclidean space $\R^{m+1} = \R \times \R^m$, the bowl soliton (cf. \cite{altwu} and \cite[Lem. 2.2]{cluschsch}) and the non-rotational entire graphs in \cite{wang_annals} for $m \ge 3$, provide examples of entire, convex graphs that are self-translators and grow of order $r^2$ at infinity. In \cite[Thm. 2.28]{bmpr}, we prove that the growth $r^2$ is sharp for a much larger class of ambient spaces.

\begin{theorem}[\cite{bmpr}, Thm. 2.28]\label{teo_soliton_intro}
	Let $M$ be a complete manifold and consider the product $\bar M^{m+1} = \R \times M$. Fix $0 \le \sigma \le 2$ and suppose that either
	\begin{equation}\label{ipo_volume_soliton_intro}
	\begin{array}{ll}
	\sigma<2 & \quad \disp \text{and} \qquad \liminf_{r \ra \infty} \frac{\log |B_r|}{r^{2-\sigma}} < \infty, \qquad \text{or} \\[0.5cm]
	\sigma = 2 & \quad \disp \text{and} \qquad \liminf_{r \ra \infty} \frac{\log|B_r|}{\log r} < \infty.
	\end{array}
	\end{equation}
	Then, there exists no graphical self-translator $u : M \ra \R$ with respect to the vertical direction $\partial_s$ satisfying 
	\begin{equation}\label{crescitav_soliton_intro}
	\begin{array}{ll}
	\disp |u(x)| = o\big( r(x)^\sigma \big) \qquad \text{as } \, r(x) \ra \infty, & \quad \text{if } \, \sigma > 0; \\[0.2cm]
	\disp u^* < \infty, & \quad \text{if } \, \sigma = 0.
	\end{array}
	\end{equation}
\end{theorem}

\begin{proof}[Proof: sketch]
	It is a direct application of Theorem \ref{teo_main_2} ($\sigma>0$) or Theorem \ref{teo_main_2_wmp} ($\sigma=0$),  taking into account that $u : M \ra \R$ satisfies
	$$
	\diver\left( \frac{D u}{\sqrt{1+|D u|^2}} \right) = \frac{1}{\sqrt{1+|D u|^2}} \qquad \text{on } \, M.
	$$
\end{proof}

The generality of Theorem \ref{teo_main_2}, and of its proof, allow for applications to self-translators in Euclidean space whose translation vector field $Y$ differs from the vertical field $\partial_s$. The next result relates to \cite{baoshi}, where the authors proved that hyperplanes are the only complete self-translators whose Gauss image lies in a spherical cap of $\Sph^{m}$ of radius $< \pi/2$: it is of interest also because it shows the sharpness of Theorem \ref{teo_main_2}, in its more general form described in Remark \ref{rem_localization}, even with respect to the constant $c$ in \eqref{eq_cuchi}.

\begin{theorem}[\cite{bmpr}, Th. 7.20]\label{teo_soliton_Rm}
	Let $\Sigma \ra \R \times \R^m$ be the entire graph over $\R^m$ associated to $v : \R^m \ra \R$. Assume that
	\begin{equation}\label{ipo_soliton_rm}
	\limsup_{r(x) \ra \infty} \frac{|v(x)|}{r(x)} = \hat v < \infty.
	\end{equation}
	Then, $\Sigma$ cannot be a self-translator with respect to any parallel field $Y$ whose angle $\vartheta \in (0, \pi/2)$ with the horizontal hyperplane $\R^m$ satisfies
	\begin{equation}\label{strano_soliton}
	\tan \vartheta > \hat v.
	\end{equation}
	In particular, if $\hat v = 0$, then $\Sigma$ cannot be a self-translator with respect to a vector $Y$ which is not tangent to the horizontal $\R^m$.
\end{theorem}

\begin{remark}
	\emph{Condition \eqref{strano_soliton} is sharp: the totally geodesic hyperplane $\{ s = x_1 \tan \vartheta \}$ is a self-translator with respect to $\partial_1 + (\tan \vartheta) \partial_s$ that realizes equality in \eqref{strano_soliton}.
	}
\end{remark}

\begin{proof}[Proof: sketch]
	Up to a rotation of coordinates on $\R^m$ and a reflection in the $\R$-axis, we can assume that $\langle Y, \partial_s \rangle > 0$ and $Y = \cos \vartheta e_1 + \sin \vartheta \partial_s$, where $e_1$ is the gradient of the first coordinate function $x_1$. By \eqref{normal_geodesic} and \eqref{prescribed_geodesic}, the soliton equation $m H = \langle Y, \nu \rangle$ satisfied by $u(x) = t(v(x)) = v(x)$ can be written as
	\begin{equation}\label{eq_soli_Rm}
	e^{-x_1 \cos \vartheta} \diver \left( e^ {x_1 \cos \vartheta} \frac{D u}{\sqrt{1+|D u|^2}} \right) = \frac{\sin \vartheta}{\sqrt{1+|D u|^2}}.
	\end{equation}
	The operator in the left-hand side is self-adjoint on $\R^m$ endowed with the weighted measure $e^{x_1 \cos \vartheta} \di x$ ($\di x$ the Euclidean volume). The proof of Theorem \ref{teo_main_2} holds verbatim in the weighted setting, replacing the Riemannian volume with the weighted volume
	$$
	\vol_{x_1 \cos \vartheta}(B_r) = \int_{B_r} e^{x_1 \cos\vartheta} \di x,
	$$
	that satisfies
	\begin{equation}\label{weightedvol}
	\liminf_{r \ra \infty} \frac{\log\vol_{x_1 \cos\theta}(B_r)}{r} = \cos \vartheta < \infty.
	\end{equation}
	Assume by contradiction that \eqref{strano_soliton} holds, in particular, $u$ is non-constant. Applying Theorem~\ref{teo_main_2} to \eqref{eq_soli_Rm} (in the strengthened form described in Remark \ref{rem_localization}) with the choices
	$$
	K = \sin \vartheta, \ \ \sigma = 1, \ \ \chi = 1, \ \ \mu = 0,
	$$
	we conclude that $\sin\vartheta \le c \hat u$, for some explicit constant $c \ge 0$. The value of $c$, that also depends on \eqref{weightedvol}, becomes $c = \cos \vartheta$. Therefore, we deduce that $\hat u \ge \tan \vartheta$, which is a contradiction.
\end{proof}

With a very similar proof, we can also consider graphical self-expanders for the MCF, i.e., graphs that move by MCF along the integral curves of the position vector field
$$
Y(\bar x) = x^j \partial_j = \frac{1}{2} \bar D |\bar x|^2 \quad \mbox{for
	all }\bar x \in \R^{m+1}.
$$

\begin{theorem}[\cite{bmpr}, Thm. 7.23]\label{teo_soliton_expa}
	There exist no entire self-expanders $\Sigma \ra \R \times \R^{m+1}$ that are the graph of a bounded function $u : \R^m \ra \R$, unless $u \equiv 0$.
\end{theorem}

\begin{proof}
	Set $\rho(x) = |x|$ for $x \in \R^m$, and write the position field along $\Sigma$ as follows:
	$$
	Y\big(v(x),x\big) = \frac{1}{2} D (\rho^2) + v \partial_s.
	$$
	The soliton equation $m H = (Y, \nu)$ satisfied by $u(x) = t(v(x)) = v(x)$ becomes
	$$
	e^{-\rho^2/2} \diver \left( e^{\rho^2/2} \frac{D u}{\sqrt{1+|D u|^2}} \right) = \frac{u}{\sqrt{1+|D u|^2}}.
	$$
	The only constant solution is $u \equiv 0$. By contradiction, assume the existence of a non-constant solution $u$. An explicit computation shows that
	$$
	\liminf_{r \ra \infty} \frac{\log \vol_{e^{\rho^2/2}}(B_r)}{r^2} < \infty.
	$$
	Applying Theorem \ref{teo_main_2}, adapted to weighted volumes, to both $u$ and $-u$, with the choices $\chi=1$, $\mu = 0$, we get $0 \le u \le 0$, which is the final contradiction.
\end{proof}

\subsection{Keller-Osserman condition and Liouville theorems}

To obtain the rigidity of solutions of the prescribed mean curvature equation without assuming an a-priori control on $u$, on general manifolds one needs $f$ to grow sufficiently fast at infinity. A notable exception is Theorem \ref{teo_tkachev} on $\R^m$, of which we now comment an improved version; in \cite[Thm. 10.30]{bmpr}, we modified the original capacity argument in \cite{Serrin_4,tkachev} to apply to solutions of
\[
\diver \left( \frac{Du}{\sqrt{1+|Du|^2}} \right) \ge \frac{f(u)}{(1+r)^\mu} \frac{|Du|}{\sqrt{1+|Du|^2}} \qquad \text{on } \, M,
\]
provided that $M$ has a polynomial volume growth, and we prove the next result. Note that the inequality corresponds to the limiting case $\chi =0$ in \eqref{eq_u_WMP}. We mention that the method is successfully used for more general equations, in particular we refer to works of Mitidieri and Pohozaev \cite{mitpoho}, D'Ambrosio and Mitidieri \cite{DAmbrMit,dambrosiomitidieri_2}, Farina and Serrin \cite{farinaserrin1,farinaserrin2}.

\begin{theorem}[Thm. 10.30 in \cite{bmpr}]\label{teo_tkachev_2}
	Let $M$ be a complete Riemannian manifold, and consider a non-decreasing $f \in C(\R)$ and a constant
	\[
	\mu < 1.
	\]
	If
	\begin{equation}\label{polinomial_tkachev}
	\liminf_{r \ra \infty} \frac{\log |B_r|}{\log r} < \infty,
	\end{equation}
	then every non-constant solution $u \in \lip_\loc(M)$ of
	\begin{equation}\label{borderline_eq_tkachev}
	\diver \left( \frac{Du}{\sqrt{1+|Du|^2}} \right) \ge \frac{f(u)}{(1+r)^\mu} \frac{|Du|}{\sqrt{1+|Du|^2}} \qquad \text{on } \, M
	\end{equation}
	satisfies $f(u) \le 0$ on $M$.
\end{theorem}

\begin{remark}[\textbf{Sharpness of $\mu<1$}]
	\emph{The restriction $\mu<1$ in the above result is sharp: example $4)$ in Remark \ref{rem_sharpness} exhibits a manifold with polynomial growth ($\alpha = -2$) satisfying
		\[
		\lim_{r \ra \infty} \frac{\log|B_r|}{\log r} = d_0 \in (2-\sigma,\infty),
		\]
		for some $\sigma \in (0,1]$, admitting a solution of
		$$
		\diver \left( \frac{D u}{\sqrt{1+|D u|^2}} \right) \ge \frac{C}{1+r} \frac{|D u|}{\sqrt{1+|D u|^2}} \qquad \text{on } \, M
		$$
		that grows like $r^\sigma$, $\sigma \in (0,1]$. Observe that the restriction $\sigma = 1$ is sharp, and indeed, the  assertion of Theorem~\ref{teo_tkachev_2} holds for every $\mu \in \R$ whenever
		\[
		\liminf_{r \ra \infty} \frac{|B_r|}{r} = 0,
		\]
		see \cite[Thm. 10.32]{bmpr}.	
	}
\end{remark}

\begin{remark}[\textbf{Sharpness of polynomial growth}]\label{rem_polysharp}
	\emph{Case $6)$ in Remark \ref{rem_sharpness}, applied with $\chi = 0$, $\mu<1$ and any $\sigma> 0$, exhibits a manifold with exponential growth satisfying
		\begin{equation}\label{eq_border}
		\lim_{r \ra \infty} \frac{\log|B_r|}{r^{1-\mu}}  \in (0, \infty),
		\end{equation}
		admitting unbounded solutions to \eqref{borderline_eq_tkachev}, for $f$ a positive constant, with arbitrarily small polynomial growth. The positivity of the limit in \eqref{eq_border} is also sharp, because otherwise, 
		by Theorem~\ref{teo_main_2} and Remark~\ref{rem_chizero}, solutions with polynomial growth to \eqref{borderline_eq_tkachev} with $f$ a positive constant do not exist.
	}
\end{remark}

As a consequence of Theorem \ref{teo_tkachev_2} and of Bernstein Theorem \ref{teo_B2}, we directly deduce the following

\begin{theorem}[\cite{cmmr}]\label{teo_cmmr_allatkachev}
	Let $M$ be a complete Riemannian manifold with $\Ricc \ge 0$. Let $f\in C(\R)$ be non-decreasing, $f \not \equiv 0$, and suppose that $b \in C(M)$ satisfies 
	\[
	b(x) \gtrsim \big( 1 + r(x) \big)^{-\mu} \qquad \text{on } \, M,
	\]	
	for some constant $\mu < 1$, where $r$ is the distance from a fixed origin. Then, any solution of
	\[
	\diver \left( \frac{Du}{\sqrt{1+|Du|^2}} \right) = b(x) f(u) \qquad \text{on } \, M,
	\]
	is a constant $c$ satisfying $f(c) = 0$.
\end{theorem}

\begin{proof}
	First, observe that $u$ solves \eqref{borderline_eq_tkachev} up to a positive constant that we can include in the definition of $f$. The Bishop-Gromov comparison theorem guarantees that \eqref{polinomial_tkachev} is satisfies, hence, by Theorem \ref{teo_tkachev_2}, $f(u) \le 0$ on $M$. Applying the same argument to $-u$ we get $f(u) \equiv 0$. Hence, $u$ gives rise to a minimal graph in $\R \times M$, that is one sided bounded because $f \not\equiv 0$. Theorem \ref{teo_graph_soloricci_intro} guarantees that $u \equiv c$ is constant, and plugging in the equation we get $f(c) = 0$.
\end{proof}
As Remark \ref{rem_polysharp} shows, when the growth of $|B_r|$ is faster than polynomial one needs some growth conditions on $f$ to guarantee that solutions of
\[
\diver \left( \frac{Du}{\sqrt{1+|Du|^2}} \right) \ge \frac{f(u)}{(1+r)^{\mu}} \frac{|Du|^{1-\chi}}{\sqrt{1+|Du|^2}} \qquad \text{on } \, M
\]
are bounded from above and satisfy $f(u^*) \le 0$. In \cite{bmpr}, for the class of inequalities
\begin{equation}\label{eq_general_phi}
\diver \left( \frac{\varphi(|Du|)}{|Du|}Du \right) \ge b(x)f(u)\ell(|Du|)
\end{equation}
and when
\[
\frac{s \varphi'(s)}{\ell(s)} \in L^1(0^+) \backslash L^1(+\infty),
\]	
considering the homeomorphism $K : [0, \infty) \ra [0, \infty)$ given by
\begin{equation}\label{def_K_gen}
K(t) = \int_0^t \frac{s \varphi'(s)}{\ell(s)}\di s,
\end{equation}
we identify the following Keller-Osserman condition:
\begin{equation}\label{eq_KO_congrad}
\int^{\infty} \frac{\di t}{K^{-1}(F(t))} < \infty,
\end{equation}	
with $F$ as in \eqref{def_F}. Observe that \eqref{eq_KO_congrad} does not depend on the growth/decay of $b$, but the latter provides the thresholds to guarantee that \eqref{eq_KO_congrad} is \emph{sufficient} to ensure that solutions of \eqref{eq_general_phi} are bounded from above and satisfy $f(u^*) \le 0$. The meaning of the bounds on $b$ is described in detail in \cite{bmpr}. However, for the mean curvature operator and in our case of interest
\begin{equation}\label{ex_l}
\ell(t) = \frac{t^{1-\chi}}{\sqrt{1+t^2}}, \qquad \chi \in [0,1],
\end{equation}
still $K(\infty) < \infty$. The proposal in \cite{maririgolisetti} to modify \eqref{def_K_gen}	for the mean curvature operator by setting
\[
K(t) = \int_0^t \frac{\varphi(s)}{\ell(s)}\di s = \int_0^t \frac{s \di s}{\ell(s)\sqrt{1+s^2}}
\]
turns out to lead to sharp results, and in our prototype case \eqref{ex_l} the associated Keller-Osserman condition \eqref{eq_KO_congrad} takes the form
\begin{equation}\label{KO_formean}
F^{\frac{1}{\chi+1}} \in L^1(\infty).
\end{equation} 	
If $f(t) \ge Ct^\omega$ for large $t$, \eqref{KO_formean} is implied by $\omega> \chi$. In this case, in \cite[Thm. 10.33]{bmpr} we prove a sharp Liouville theorem on manifolds with controlled volume growth of geodesic balls.

\begin{theorem}[\cite{bmpr}, Thm 10.33]\label{teo_main}
	Let $M$ be a complete Riemannian manifold, and fix parameters
	\[
	\chi \in [0,1], \qquad \mu \le \chi+1.
	\]
	Let $f \in C(\R)$ satisfy
	\[
	f(t) \ge C t^\omega \qquad \text{for } \, t >>1,
	\]
	for some $\omega > \chi$, and assume that $u \in \lip_\loc(M)$ is a non-constant solution of
	\begin{equation}\label{eq_PMC}
	\diver \left( \frac{Du}{\sqrt{1+|Du|^2}} \right) \ge \frac{f(u)}{(1+r)^\mu}\frac{|Du|^{1-\chi}}{\sqrt{1+|Du|^2}} \qquad \text{on } \, M.
	\end{equation}
	If either
	\begin{equation}\label{volgrowth_sigmazero_Linfty_inthetheorem}
	\begin{array}{lll}
	\mu < \chi+1 & \text{and} & \disp \qquad \liminf_{r \ra \infty} \frac{\log |B_r|}{r^{\chi+1-\mu}} < \infty \quad \text{($=0$ if $\chi=0$)}; \\[0.4cm]
	\text{or}\\[0.1cm]
	\mu = \chi+1 & \text{and} & \disp \qquad \liminf_{r \ra \infty} \frac{\log|B_r|}{\log r} < \infty \quad \text{($\le 2$ if $\chi=0$)},
	\end{array}
	\end{equation}
	then $u$ is bounded above and $f(u^*) \le 0$.
\end{theorem}

\begin{remark}
	\emph{Results covering some of the cases of Theorem \ref{teo_main} were obtained, in Euclidean space, for the stronger inequality
		\[
		\diver \left( \frac{Du}{\sqrt{1+|Du|^2}} \right) \ge \frac{f(u)}{(1+r)^\mu}|Du|^{1-\chi} \qquad \text{on } \, \R^m
		\]
		assuming $\omega > \chi$. Among them, we quote \cite[Cor. 1.3,1.4]{filippucci} by Filippucci, for $\chi \in (0,1]$ and $\mu<\chi+1$ under the restriction\footnote{The bound $2 \in (1,m)$ is assumed at \cite[p.2904]{filippucci}.} $m>2$, later improved in \cite[Thm.1]{farinaserrin2} to cover the range\footnote{According to \cite{farinaserrin2}, the case $\mu = \chi+1$ would need a further requirement that, for the mean curvature operator, reads as $2>m$, contradiction.} $\chi \in [0,1]$, $\mu < \chi+1$. Rigidity for $\chi = 1$ and $\mu = \chi+1=2$ on $\R^m$ was conjectured by Mitidieri-Pohozaev in \cite[Sect. 14 Ch. 1]{mitpoho}, and proved by Usami \cite{usami}.
	}
\end{remark}

\begin{remark}[\textbf{Sharpness of $\omega> \chi$}]\label{rem_sharpness_2}
	\emph{We consider the radially symmetric model in Subsection \ref{subsec_WSMP}. Because of the asymptotic growth of $\log |\BB_r|$ in \eqref{crescitevol_counter}, given $\chi \in [0,1]$, the condition
		\[
		\mu < \chi + 1, \qquad \liminf_{r \ra \infty} \frac{\log |\BB_r|}{r^{\chi+1-\mu}} < \infty \quad \text{($=0$ if $\chi=0$)}
		\]
		is equivalent to
		\[
		\alpha > -2, \qquad \left\{ \begin{array}{ll}
		\mu \le \chi - \dfrac{\alpha}{2} & \quad \text{if } \, \chi > 0 \\[0.4cm]
		\mu < \chi - \dfrac{\alpha}{2} & \quad \text{if } \, \chi = 0,
		\end{array}\right.
		\]		
		while, for $\chi \in (0,1]$, condition	
		\[
		\mu = \chi + 1, \qquad \liminf_{r \ra \infty} \frac{\log |\BB_r|}{\log r} < \infty
		\]
		is equivalent to
		\[
		\alpha = -2, \quad \mu = \chi - \dfrac{\alpha}{2}.
		\]
		The case $\chi = 0$ in the second of \eqref{volgrowth_sigmazero_Linfty_inthetheorem} will not be considered in the present counterexample. Keeping this in mind, we prove the optimality of condition $\omega> \chi$. The computations in Subsection 10.4 of \cite{bmpr} (cf. also Remark 10.29 therein) show that the radially symmetric function constructed in Remark \ref{rem_sharpness}, that is positive and coincides with $r^\sigma$ for large $r$, solves
		\begin{equation}\label{asd_2}
		\diver \left( \frac{Du}{\sqrt{1+|Du|^2}} \right) \gtrsim \frac{u^\omega}{(1+r)^\mu}\frac{|Du|^{1-\chi}}{\sqrt{1+|Du|^2}} \qquad \text{on } \, M
		\end{equation}
		provided that $\chi \in [0,1]$ and either
		\begin{equation}\label{bonito_mc}
		\left\{ \begin{array}{lll}
		\omega< \chi, & \quad \text{for each} & \mu \le \chi - \dfrac{\alpha}{2}, \quad \text{or } \\[0.3cm]
		\omega = \chi & \quad \text{and} & \mu = \chi- \dfrac{\alpha}{2},
		\end{array}\right.
		\end{equation}
		Hence, $\omega > \chi$ is necessary for the conclusion in Theorem \ref{teo_main}. The last restriction in the second of \eqref{bonito_mc} is also optimal: in the ``Euclidean setting" $\alpha = -2$, if $\omega = \chi$ and $\mu < \chi+1$ then entire solutions of \eqref{asd_2}, with the equality sign, are constant if they have polynomial growth, see \cite[Thm. 12]{farinaserrin2} and also Example 4 at p. 4402 therein.
	}
\end{remark}

To extend the above theorem to more general $f$ satisfying the Keller-Osserman condition \eqref{KO_formean}, we had to use a different approach inspired by the classical Phr\'agmen-Lindel\"off method, see \cite[sec. 10.3]{bmpr}: under the validity of \eqref{KO_formean}, the key point is to construct a suitable family of radial blowing-up solutions $\{w_\eps\}_\eps$ of
\[
\diver \left( \frac{Dw_\eps}{\sqrt{1+|Dw|^2}} \right) \le \frac{f(w_\eps)}{(1+r)^\mu}\frac{|Dw_\eps|^{1-\chi}}{\sqrt{1+|Dw_\eps|^2}} \qquad \text{on } \, M
\]
that locally converge to a constant function as $\eps \ra 0$. The family $\{w_\eps\}_{\eps}$ is used to constrain $u$ via the use of a comparison principle. The presence of the gradient term in the RHS of \eqref{eq_PMC} requires to control the gradient of $u$ and $w_\eps$ in a neighbourhood of the set of maxima of $u-w_\eps$, and therefore we need to restrict ourselves to $u \in C^1(M)$.\par
The construction of $w_\eps$ is rather delicate, especially in the effort to avoid unnecessary technical conditions. It requires to estimate $\Delta r$ from above, and therefore, by comparison theory, it needs a  control on the Ricci curvature from below. We require the same bounds as those in \eqref{teo_SMP_intro}, and prove the following result that, because of the discussion at the end of Subsection~\ref{subsec_WSMP}, well fits 
with Theorem~\ref{teo_main}.

\begin{theorem}[\cite{bmpr}, Thm. 10.26]\label{teo_SMP_SL_mc}
	Let $M$ be a complete $m$-dimensional manifold satisfying
	\[
	\Ricc (\nabla r, \nabla r) \ge -(m-1)\kappa^2\big( 1+r^2\big)^{\alpha/2} \qquad \text{on } \, M \backslash \cut(o),
	\]
	for constants $\kappa \ge 0, \alpha \ge -2$, with $r$ the distance from a fixed origin $o$. Choose
	\begin{equation}\label{mualphachi12_mc}
	\chi \in (0,1], \qquad \mu \le \chi - \frac{\alpha}{2},
	\end{equation}
	and let $f \in C(\R)$ be non-decreasing. Then, under the validity of the Keller-Osserman condition
	\begin{equation}\label{KO_mc}
	F^{-\frac{1}{\chi+1}} \in L^1(\infty),
	\end{equation}
	with $F$ as in \eqref{def_F}, any non-constant solution $u \in C^1(M)$ of
	\[
	\diver \left( \frac{Du}{\sqrt{1+|Du|^2}} \right) \ge \frac{f(u)}{(1+r)^\mu}\frac{|Du|^{1-\chi}}{\sqrt{1+|Du|^2}} \qquad \text{on } \, M
	\]
	is bounded above and satisfies $f(u^*) \le 0$.
\end{theorem}

\subsection{Application: the capillarity equation}

Applying Theorem \ref{teo_main} to the capillarity problem, we deduce the following corollary. To our knowledge, this is the first result on the capillarity equation that allows $M$ to have geodesic balls with growth of their volume faster than polynomial.

\begin{theorem}\label{teo_capillarity_pre}
	Let $M$ be complete. For a fixed origin $o$, denote with $r(x) = \mathrm{dist}(x,o)$ and with $B_r$ the geodesic ball centered at $o$. Let $\kappa \in C(M)$ satisfy
	\begin{equation}\label{ipo_kappa_pre}
	\kappa(x) \gtrsim \big( 1+ r(x)\big)^{-\mu} \qquad \text{for all } \, x \in M,
	\end{equation}
	for some $\mu < 2$. If there exists $\eps>0$ such that
	\[
	\liminf_{r \ra \infty} \frac{\log |B_r|}{r^{2-\eps-\mu}} < \infty,
	\]
	then the only solution of the capillarity equation
	\[
	\diver \left( \frac{D u}{\sqrt{1+|D u|^2}} \right) = \kappa(x)u \qquad \text{on } \, M
	\]
	is $u \equiv 0$.
\end{theorem}

\begin{proof}[Proof: sketch]
	Evidently, $u \equiv 0$ is the only constant solution. Assume the existence of a non-constant solution $u$. Up  replacing $u$ with $-u$, we can assume that $\{u>0\} \neq \emptyset$. By assumption, there exists $\chi < 1$ sufficiently close to $1$ in such a way that
	\[
	\liminf_{r \ra \infty} \frac{\log |B_r|}{r^{1+\chi-\mu}} < \infty.
	\]
	The capillarity equation implies
	\[
	\diver \left( \frac{D u}{\sqrt{1+|D u|^2}} \right) = \kappa(x)u \gtrsim \frac{u}{(1+r)^\mu} \frac{|Du|^{1-\chi}}{\sqrt{1+|Du|^2}} \qquad \text{on } \, \{u > 0\}.
	\]
	Therefore, $u_+ : = \max\{u,0\}$ is a non-constant solution of
	\[
	\diver \left( \frac{D u_+}{\sqrt{1+|D u_+|^2}} \right) \gtrsim \frac{u_+}{(1+r)^\mu} \frac{|Du_+|^{1-\chi}}{\sqrt{1+|Du_+|^2}} \qquad \text{on } \, M,
	\]
	where we used that $t^{1-\chi}/\sqrt{1+t^2}$ is bounded on $\R$. The inclusion of the ``artificial" gradient term enables the function $f(u_+) = u_+$ to satisfy the Keller-Osserman condition $1 > \chi$, and therefore we can apply Theorem \ref{teo_main} to obtain $u_+^* \le 0$, contradiction.
\end{proof}

\section{Gradient estimates and Bernstein theorems in $\R \times M$}\label{sec_gradient}

As we saw above, under general assumptions on $b,f,\ell$ solutions of 
\[
\diver \left( \frac{Du}{\sqrt{1+ |Du|^2}}\right) = b(x)f(u)\ell(|Du|) \qquad \text{on } \, M
\]
indeed satisfy $f(u) \equiv 0$, in particular 
\begin{equation}\label{eq_minimal_Rm_2}
\diver \left( \frac{Du}{\sqrt{1+ |Du|^2}}\right) = 0 \qquad \text{on } \, M.
\end{equation}
However, to infer the constancy of $u$ solving \eqref{eq_minimal_Rm_2}, one needs rather different arguments (and more binding assumptions) than those leading to the results in the previous section; indeed, while the above theorems apply to differential \emph{inequalities}, those in this section are very specific to the equality case, unless in the special situation when $M$ has slow volume growth. This consideration is not surprising, as it parallels the case of harmonic functions: positive solutions of $\Delta u = 0$ are constant on each complete manifold with $\Ricc \ge 0$ by \cite{yau,chengyau}, while the constancy of every positive solution of $\Delta u \le 0$ is equivalent to the parabolicity of $M$ (cf. \cite{schoenyau, grigoryan}) that, for complete manifolds with $\Ricc \ge 0$, is equivalent to the slow volume growth condition 
\begin{equation}\label{eq_varopoulos}
\int^{\infty} \frac{s \di s}{|B_s|} = \infty.
\end{equation}
As a matter of fact, \eqref{eq_varopoulos} is sufficient for the parabolicity of a complete manifold $M$, regardless to any curvature condition (cf. \cite[Thm. 7.5]{grigoryan}). The next example is instructive:
\begin{example}
	Consider the warped product manifold 
	\[
	M = \R \times_\xi N, \qquad \text{with $N$ compact and } \, \xi(t) = \sqrt{t^2+2}.
	\]
	Let $\varphi(t) = t/\sqrt{1+t^2}$. Then, a direct check shows that 
	\[
	u(t,y) = \int_0^t \varphi^{-1} \left( \frac{1}{\xi(\tau)^{m-1}} \right) \di \tau
	\]
	is well defined on $M$ (since $\xi > 1$ on $\R$) and solves \eqref{eq_minimal_Rm}. Denoting with $r$ the distance from an origin $(0,y)$, observe that
	\begin{equation}\label{example_mingraph}
	\Sect \ge -\frac{C}{1+r^2}, \qquad |B_r| \le C r^{m}.
	\end{equation}
	for some constant $C>0$. Furthermore, $u$ is bounded on $M$ if and only if $m \ge 3$. Therefore, if $m \ge 3$, assumptions \eqref{example_mingraph} are not sufficient to deduce the constancy of positive solutions of \eqref{eq_minimal_Rm_2}, not even of bounded ones. As shown in \cite{chh_nonexistence}, if the first in \eqref{example_mingraph} is strengthened to $M$ having asymptotically non-negative Sectional curvature, that is, if 
	\begin{equation}\label{asnonneg_sectional}
	\Sect \ge - G(r) \qquad \text{with $0 < G \in C(\R^+_0)$ satisfying} \qquad \int^{+\infty} r G(r) \di r < \infty, 
	\end{equation}	
	and if $M$ has only one end, then positive solutions of \eqref{eq_minimal_Rm_2} that are bounded on one side are constant. 	
\end{example}

The situation for manifolds with slow volume growth, including the above example when $m =2$, is quite rigid. Suppose that $u$ is a positive solution of \eqref{eq_minimal_Rm_2} on $M$, and assume that $M$ satisfies \eqref{eq_varopoulos}. Let $g$ be the metric on the graph $\Sigma$. By a calibration argument that can be found in \cite{liwang_mini}, the volume of the portion of $(\Sigma, g)$ inside an ambient geodesic ball $\BB_r \subset M \times \R$ centered at $(o, u(o)) $ satisfies
\begin{equation}\label{eq_transplant_volume}
|\Sigma \cap \BB_r| \le |B_r| + \frac{1}{2} |B_{2r}\backslash B_r| \le 2|B_{2r}|,	
\end{equation}
with $B_s$ the geodesic ball of radius $s$ in $(M, \sigma)$ centered at $o$. In particular, since the geodesic ball $B_r^g$ in $(\Sigma, g)$ is contained in $\Sigma \cap \BB_r$, its volume satisfies \eqref{eq_varopoulos}. Hence, $\Sigma$ is parabolic. The constancy of $u$ follows since $u>0$ solves $\Delta_g u = 0$ on $\Sigma$, with $\Delta_g$ the Laplace-Beltrami operator of $g$. Summarizing, 
\[
\begin{array}{l}
\text{A complete manifold with slow volume growth, in the sense of \eqref{eq_varopoulos}, satisfies \eqref{bernstein_2}:}\\
\text{entire positive minimal graphs over $M$ are constant.}
\end{array}
\] 		
Inequality \eqref{eq_transplant_volume} allows to ``transplant" a volume condition on $M$ to a corresponding one on $\Sigma$, in particular condition \eqref{eq_varopoulos} that suffices for parabolicity, without any a-priori bounds on $u$. A natural question that arises is whether there is a way to transplant the parabolicity itself from $M$ to $\Sigma$. Currently, we are not aware of any example of entire, minimal graph over a complete, parabolic manifold that is not parabolic.


If the volume growth of $(M,\sigma)$ does not satisfy \eqref{eq_varopoulos}, the main tool to prove the constancy of solutions of \eqref{eq_minimal_Rm_2} is a local gradient bound for positive solutions $u : B_r \subset M \ra \R^+$, as in \cite{bdgm}. However, the possible lack of a Sobolev inequality on manifolds satisfying $\Ricc \ge -(m-1)\kappa^2$ (or $\Sect \ge - \kappa^2$) forced to devise a method different from the one in \cite{bdgm}. Typically, in a manifold setting the estimate is achieved by refining the classical idea of Korevaar \cite{korevaar}, as for instance in \cite{rosenbergschulzespruck}. There, the authors inferred the following bound for positive solutions $u : B_r \subset M \ra \R$ under the requirements that $\Ricc \ge 0$ and $\Sect \ge - \kappa^2$ on $B_r$:
\[
|Du(x)| \le c_1\exp\left\{ c_2[1 + \kappa r \coth(\kappa r)] \frac{u^2(x)}{r^2} \right\}, 
\]
for some $c_j=c_j(m)$. Further improvements can be found in \cite{dl_bounded,dl_unbounded,chh_nonexistence}. The proof of the constancy of $u$ proceeds by first showing the global bound $|Du| \le C$ on $M$, for some constant $C>0$, that allows to interpret the mean curvature operator as a uniformly elliptic linear operator on $(M,\sigma)$ acting on $u$. At this stage, for instance, the quadratic lower bound \eqref{example_mingraph} on $\Sect$ suffices to show $|Du| \le C$ for every positive minimal graph with at most linear growth, \cite{chh_nonexistence}. Next, assumptions like 
\[
\left\{\begin{array}{l}
\Ricc \ge 0, \quad \text{or} \\
\Sect \ \text{ satisfying \eqref{asnonneg_sectional}, and $M$ possessing only one end, as in \cite{chh_nonexistence}}
\end{array}\right.
\]
serve to guarantee the validity of the global Harnack inequality $\sup_M u \le C' \inf_M u$, from which the constancy of $u$ follows. We mention that, recently, in \cite{ding} the author proved Theorem \ref{teo_B2} on each manifold supporting global doubling and a weak Neumann-Poincar\'e inequalities, a class that includes manifolds with $\Ricc \ge 0$. His idea is to avoid the need of gradient estimates by transplanting the doubling and weak Neumann-Poincar\'e inequalities (the latter, only for suitable functions) from $M$ to the graph $\Sigma$, and then proving the Harnack inequality for the harmonic function $u$ on $\Sigma$.\par

\subsection{Ricci curvature and gradient estimate for minimal graphs}	

Our approach in \cite{cmmr} to obtain Theorem \ref{teo_B2} still relies on a gradient estimate. Precisely, we prove 

\begin{theorem}[\cite{cmmr}]\label{teo_graph_soloricci_intro}
	Let $M$ be a complete manifold of dimension $m \ge 2$ satisfying 
	\begin{equation}\label{eq_lowerricci}
	\Ricc \ge - (m-1)\kappa^2, 
	\end{equation}
	for some constant $\kappa \ge 0$. Let $\Omega \subset M$ be an open subset and let $u \in C^\infty(\Omega)$ be a positive solution of \eqref{eq_minimal_Rm_2} on $\Omega$. If either
	\begin{itemize}
		\item[$(i)$] $\Omega$ has locally Lipschitz boundary and 
		\[
		\liminf_{r \ra \infty} \frac{\log |\partial \Omega \cap B_r|}{r^2} < \infty, \qquad \text{or}
		\]
		\item[$(ii)$] $u \in C(\overline \Omega)$ and is constant on $\partial \Omega$.
	\end{itemize}
	Then 
	\begin{equation}\label{eq_bellissima}
	\frac{\sqrt{1+|Du|^2}}{e^{\kappa u\sqrt{m-1}}} \leq  \max\left\{ 1, \limsup_{x \ra \partial \Omega} \frac{\sqrt{1+|Du(x)|^2}}{e^{\kappa u(x)\sqrt{m-1}}}\right\} \qquad \text{on } \, \Omega.
	\end{equation}
	In the particular case $\Omega = M$, 
	\begin{equation}\label{eq_entire}
	\sqrt{1+|Du|^2} \le e^{\kappa u\sqrt{m-1}} \qquad \text{on } \, M.
	\end{equation}
	If equality holds in \eqref{eq_entire} at some point, then $\kappa = 0$ and $u$ is constant. 	
\end{theorem}

Despite we found no explicit example, we feel likely that the bound \eqref{eq_entire} be sharp also for $\kappa > 0$, in the sense that the constant $\kappa \sqrt{m-1}$ cannot be improved. Our estimate should be compared to the one for positive harmonic functions on manifolds satisfying \eqref{eq_lowerricci}, obtained by P. Li and J. Wang \cite{liwang} by refining Cheng-Yau's argument:
\begin{equation}\label{eq_chengyau}
|Du| \le (m-1)\kappa u \qquad \text{on } \, M,
\end{equation}
and its version for sets with boundary in \cite[Thm. 2.24]{maririgolisetti_mono}:
\begin{equation}\label{eq_chengyau_boundary}
\frac{|D u|}{(m-1)\kappa u} \le \max \left\{ 1, \limsup_{x \ra \partial \Omega} \frac{|D u(x)|}{(m-1)\kappa u(x)} \right\} \qquad \text{on } \, \Omega. 
\end{equation}
Indeed, to reach our goal we reworked Korevaar's technique, that parallels Cheng-Yau's estimate though with some notable difference. We also tried to simplify the method and highlight its key features, in such a way to obtain neat constants in \eqref{eq_bellissima}. Writing $\nabla, \|\cdot \|$ for the connection and norm in the graph metric $g$ on $\Sigma$, and letting 
\[
{\bf n} : = \frac{\partial_s - Du}{W}, \qquad W : = \sqrt{1+|Du|^2}
\]
denote the upward pointing unit normal of $\Sigma$, the Jacobi equation for $W^{-1} = \langle {\bf n}, \partial_s \rangle$ leads to the identity
\[
\LL_g W = \Big( \|\II_\Sigma\|^2 + \overline{\Ricc}({\bf n},{\bf n}) \Big) W, 
\]
where $\|\II_\Sigma\|^2$ is the second fundamental form of $\Sigma$, $\overline{\Ricc}$ is the Ricci curvature of $M \times \R$, and $\LL_g$ is the weighted Laplacian 
\begin{equation}\label{def_Lg}
\LL_g \phi = \Delta_g \phi - 2\langle \frac{\nabla W}{W}, \nabla \phi) = W^2{\rm div}_g \big(W^{-2} \nabla \phi \big).
\end{equation}
The lower bound on the Ricci curvature implies
\[
\Ricc({\bf n},{\bf n}) \ge - (m-1)\kappa^2 \frac{|Du|^2}{W^2} = -(m-1) \kappa^2 \|\nabla u\|^2,
\]
so the negativity of $\Ricc$ can be absorbed by defining 
\[
z : = W e^{-Cu}, \qquad \text{for} \quad C \ge \kappa \sqrt{m-1},
\]
and noting that $z$ satisfies 
\begin{equation}\label{eq_z}
\LL_g z \ge \|\II_\Sigma\|^2z + \big( C^2 - \kappa^2(m-1)\big)\|\nabla u\|^2 \ge \|\II_\Sigma\|^2z.
\end{equation}
However, from \eqref{eq_z} it is not clear that $z$ satisfies a maximum principle on $\Omega$ if $\Omega$ is unbounded. Note that \eqref{eq_z} somehow corresponds, in the case of harmonic functions, to the equation for $\bar z = |Du|/u = |D\log u|$ (cf. \cite{liwang}):
\begin{equation}\label{eq_cylw}
\LL_\sigma' \bar z \ge \disp \frac{\bar z^3}{m-1} - (m-1)\kappa^2 \bar z : = f(\bar z).
\end{equation}
for a suitable weighted Laplacian $\LL_\sigma'$. The term with $\bar z^3$ appears because of the refined Kato inequality 
\[
|D^2 u|^2 \ge \frac{m}{m-1}|D|Du||^2 
\]
for solutions of $\Delta u = 0$, and allows $f$ to satisfy the Keller-Osserman condition \eqref{KO}. Although Cheng-Yau-Li-Wang's method proceeds by localizing \eqref{eq_cylw} with suitable cut-offs, the validity of \eqref{eq_chengyau}, equivalently, $f(\sup_M \bar z) \le 0$, is expectable also in view of Theorem \ref{teo_main} above. Indeed, integral estimates close to those in Theorem \ref{teo_main} were used, in place of localization tecniques, to prove the generalization of \eqref{eq_chengyau} and \eqref{eq_chengyau_boundary} for $p$-harmonic functions without the need to bound the sectional curvature of $M$, see \cite{sungwang,maririgolisetti_mono}.\par
On the contrary, in \eqref{eq_z} it seems difficult to relate, if possible, the refined Kato inequality for $\|\nabla^2 u\|^2 = W^{-2} \|\II_\Sigma\|^2$ to $z$ in order that the resulting $f$ matches \eqref{KO}. Therefore, to overcome the problem we localize as in Korevaar's method: 

\begin{proof}[Sketch of the proof of Theorem \ref{teo_graph_soloricci_intro}]
	Fix 
	\[
	C> \kappa \sqrt{m-1}.
	\]
	We claim that the following set is empty for every choice of constant $\tau>0$:
	\[
	\Omega'' : = \left\{ x \in \Omega \ : \ z(x) > \max\left\{ 1, \limsup_{y \ra \partial \Omega} \frac{W(y)}{e^{\kappa \sqrt{m-1}u(y)}} \right\} + \tau \right\}.	
	\]
	Once the claim is shown, \eqref{eq_bellissima} follows by letting $\tau \ra 0$ and then $C \downarrow \kappa \sqrt{m-1}$. By contradiction, assume that $\Omega'' \neq \emptyset$, and let $\delta, \eps', \eps$ be positive small numbers to be chosen later. Consider a non-negative exhaustion $\varrho$ on $\overline{\Omega''}$ (that is, $\{\varrho \le c\}$ is compact in $\overline{\Omega''}$ for every constant $c$), and define
	\[
	z_0 = W \left( e^{-Cu - \eps \varrho} - \delta \right) : = W \eta
	\]
	Since $u>0$, observe that $\{z_0>0\}$ is relatively compact in $\overline{\Omega''}$. For $\eps, \delta$ small enough, 
	\[
	\Omega' : = \left\{ x \in \Omega \ : \ z_0(x) > \max\left\{ 1, \limsup_{y \ra \partial \Omega} \frac{W(y)}{e^{\kappa \sqrt{m-1}u(y)}} \right\} + \tau \right\} \subset \Omega''.	
	\]
	is non-empty. Computing $\LL_g z_0$ we get
	\begin{equation}\label{eq_z0}
	\LL_g z_0 \ge W \Big[ \Delta_g \eta - (m-1)\kappa^2 \|\nabla u\|^2 \eta\Big] \qquad \text{on } \, \Omega',
	\end{equation}
	and moreover 
	\begin{equation}\label{eq_deltaeta}
	\begin{array}{lcl}
	\disp \frac{\Delta_g \eta}{\eta + \delta} & = & \disp \disp \| C \nabla u + \eps \nabla \varrho\|^2 - \eps \Delta_g \varrho \\[0.2cm]
	& \ge & \disp C^2(1-\eps')\|\nabla u\|^2 - \eps \left[ \Delta_g \varrho + \eps\left(1+ \frac{1}{\eps'}\right)\|\nabla \varrho\|^2\right],
	\end{array}  	
	\end{equation}
	where we used Young inequality with free parameter $\eps'>0$ to be specified later. In \cite{rosenbergschulzespruck,chh_nonexistence,dl_bounded,dl_unbounded}, at this stage $\varrho$ is chosen to be the distance function $r$ on $(M,\sigma)$ from a fixed origin. Since the components $g^{ij}$ and $\sigma^{ij}$ of $g^{-1}$ and $\sigma^{-1}$ are related by 
	\[
	g^{ij} = \sigma^{ij} - \frac{u^iu^j}{W^2},
	\]
	then we deduce $\|\nabla r\|^2 \le |Dr|^2 = 1$, while because of minimality
	\[
	\Delta_g r = g^{ij}r_{ij},  
	\]
	with $r_{ij}$ the components of $D^2r$. Since the eigenvalues of $g^{ij}$ with respect to $\sigma$ are $1$ and $W^{-2}$, that might be very different, to estimate $\Delta_g r$ an information on the full Hessian of the distance function (hence, an assumption involving $\Sect$) seems unavoidable. Indeed, one can weaken the bound on $\Sect$ to a bound involving an intermediate Ricci curvature, but a control just on $\Ricc$ does not suffice. Our first observation is that, in \eqref{eq_deltaeta}, the quantities $\Delta_g \varrho$ and $\|\nabla \varrho\|^2$ need not be estimated separately, but the last term in the right-hand side will be small for suitably small $\eps << \eps'$ provided that 
	\begin{equation}\label{eq_varrho}
	\Delta_g \varrho \le 1 - \|\nabla \varrho\|^2 \qquad \text{on } \, \Omega'
	\end{equation}
	We recognize that \eqref{eq_varrho} is satisfied by the logarithm of a solution $v$ of 
	\begin{equation}\label{eq_khasminskii}
	\left\{ \begin{array}{l}
	\Delta_g v \le v, \qquad v>1 \qquad \text{on } \, \Sigma \\[0.2cm]
	v \ \ \text{ is an exhaustion}, 
	\end{array}\right.
	\end{equation}
	which makes contact with the equations studied in the previous section, and in fact suggests to produce $v$ by mean of potential theory. We use the information that $\Sigma$ is calibrated, so one can control the volume of balls in $\Sigma$ in terms of the measures of $\Omega$ and $\partial \Omega$ in $M$: by modifying the method leading to \eqref{eq_transplant_volume}, we prove in \cite[Lem. 1.2]{cmmr} that $\Omega''$ can be isometrically embedded as an open subset of a complete manifold $(N,h)$ whose geodesic balls $B_r^h$ satisfy, under either of assumptions $(i)$ and $(ii)$ in Theorem \ref{teo_graph_soloricci_intro},
	\begin{equation}\label{eq_integral}
	\liminf_{r \ra \infty} \frac{\log |B_r^h|}{r^2} < \infty.
	\end{equation}
	Therefore, by \cite[Thm.4.1]{prsmemoirs} the weak maximum principle holds for $\Delta_h$ on $(N,h)$, namely, for $f \in C(\R)$ every solution of $\Delta_h u = f(u)$ that is bounded from above satisfies $f(u^*) \le 0$. In particular, the only non-negative solution of 
	\[
	\Delta_h \phi \ge \phi \qquad \text{with } \, \sup_N \phi < \infty
	\]
	is $\phi \equiv 0$ 
	
	\begin{remark}
		\emph{By \cite{prsmemoirs}, the validity of the weak maximum principle for $\Delta_h$ turns out to be equivalent to the stochastic completeness of $(N,h)$, that is, to the property that the minimal Brownian motion on $N$ is non-explosive. Since \eqref{eq_integral} implies
			\[
			\int^\infty \frac{s \, \di s}{\log |B^h_s|} = \infty,
			\]
			then the stochastic completeness of $N$ also follows by \cite[Thm. 9.1]{grigoryan}. 
		}
	\end{remark}
	
	Next, we use the AK-duality principle (Ahlfors-Khas'minskii duality) recently established in \cite{marivaltorta,maripessoa,maripessoa_2}: in our case of interest, the AK-duality guarantees that the weak maximum principle is equivalent to the possibility to produce solutions of \eqref{eq_khasminskii} on $(N,h)$. We mention that the AK-duality works in greater generality for subsolutions of fully nonlinear PDEs of the type, 
	\begin{equation}\label{eq_subsol_F}
	\mathscr{F}(x,u,\di u, \nabla^2 u) \ge 0,
	\end{equation}
	under mild assumptions on $\mathscr{F}$, and states the equivalence between a maximum principle on unbounded open sets for solutions of \eqref{eq_subsol_F}, and the existence of suitable families of exhaustions $v$ that solve
	\[
	\mathscr{F}(x,v,\di v, \nabla^2 v) \le 0.
	\]
	While, for general $\mathscr{F}$, the construction of $v$ is subtle and proceeds by stacking solutions of obstacle problems, in our case the linearity of $\Delta_h$ allows to produce $v$ in a much simpler way: fix a small relatively compact ball $B$ and an exhaustion $\Omega_j \uparrow N$ by means of relatively compact, smooth open sets, and for each $j$ let $z_j$ solve
	$$
	\left\{ \begin{array}{l}
	\Delta_h z_j = z_j \qquad \text{on } \, \Omega_j\backslash B, \\[0.2cm]
	z_j = 0 \quad \text{on } \, \partial B, \qquad z_j = 1 \quad \text{on } \, \partial \Omega_j.
	\end{array}\right.
	$$
	The maximum principle implies that $0 < z_j < 1$ on $\Omega_j \backslash \overline{B}$, thus setting $z_j=1$ on $N \backslash \Omega_j$ produces a locally Lipschitz, viscosity solution of $\Delta_h z_j \le z_j$ on $N \backslash \overline{B}$. By elliptic estimates, $z_j \downarrow z$ locally smoothly on $N \backslash B$, for some $0 \le z \in C^\infty(N \backslash B)$ solving  
	$$
	\left\{ \begin{array}{l}
	\Delta_h z = z \qquad \text{on } \, N \backslash \overline{B}, \\[0.2cm]
	z = 0 \quad \text{on } \, \partial B.
	\end{array}\right.
	$$
	Since $(N,h)$ is stochastically complete, necessarily $z \equiv 0$. Thus, given $k \in \mathbb{N}$ there exists $j(k)$ such that $z_{j(k)} < 2^{-k}$ on $\Omega_k$. The series 
	$$
	v : = \sum_{k=1}^\infty z_{j(k)}
	$$
	converges locally uniformly to a locally Lipschitz exhaustion satisfying $\Delta_h v \le v$ on $N \backslash B$. The full set of properties in \eqref{eq_khasminskii} can be matched by extending $v$ in $B$, then suitably translating and rescaling $v$, and eventually smoothing $v$ by using Greene-Wu approximation techniques \cite{GW}.\par
	Having shown the existence of $\varrho$ in \eqref{eq_varrho}, inserting into \eqref{eq_deltaeta} and using that 
	\[
	\|\nabla u\|^2 = \frac{W^2 - 1}{W^2} \ge 1 - \frac{1}{(1+\tau)^2}
	\]
	We infer the existence of $\eps'$ sufficiently small, and of $\eps(\eps',\tau) << \eps'$ small enough that
	\[
	\Delta_g \eta \ge C^2(1- 2\eps')\|\nabla u\|^2 (\eta+ \delta), 
	\]
	and thus, up to further reducing $\eps'$, we conclude from \eqref{eq_z0} 
	\[	
	\LL_g z_0 > 0 \qquad \text{on } \, \Omega',
	\]
	a contradiction that concludes the proof.
\end{proof}

\subsection{Ricci curvature and gradient estimate for CMC graphs}

The above method allows for generalizations to the CMC case, obtained in the very recent \cite{cmmr_2}, that apply to the rigidity of capillary graphs over unbounded regions. We detail the application in the next section. Although the guiding idea is the same as the one for Theorem \ref{teo_graph_soloricci_intro}, the difficulty to deal with nonzero $H$ makes the statement of the next result, and its proof, more involved.   

\begin{theorem}[\cite{cmmr_2}] \label{thm-CMCbound}
	Let $M$ be a complete manifold of dimension $m \geq 2$ satisfying
	$$
	\Ricc \ge - (m-1)\kappa^2 \qquad \text{on } \, M
	$$
	for some constant $\kappa \ge 0$. Let $\Omega \subseteq M$ be an open subset and let $u \in C^\infty(\Omega)$ be a positive solution of 
	\[
	\diver \left( \frac{Du}{\sqrt{1+|Du|^2}}\right) = mH \qquad \text{on } \, \Omega, 
	\]
	for some $H \in \R$. Suppose that either (i) or (ii) of Theorem \ref{teo_graph_soloricci_intro} are satisfied.  Let $C \geq 0$ and $A \geq 1$ satisfy
	\begin{itemize}
		\item[-] if $H< 0$, 
		\begin{equation}\label{C1HA0}
		mH^2 + C^2 - (m-1)\kappa^2 \ge 0
		\end{equation}
		\item[-] if $H \ge 0$,
		\begin{align}
		\label{CAH1}
		& mH^2 - \frac{CmH}{A} + (C^2 - (m-1)\kappa^2 ) \frac{A^2 - 1}{A^2} \ge 0, \\
		\label{CAH2}
		& mHC A + 2(C^2 - (m-1)\kappa^2) \ge 0.
		\end{align}
	\end{itemize}
	Then,
	\begin{equation} \label{CMC-gen-bound}
	\frac{\sqrt{1+|Du|^2}}{e^{Cu}} \leq \max \left\{ A, \limsup_{y\to\partial\Omega} \frac{\sqrt{1+|Du(y)|^2}}{e^{C u(y)}} \right\} \qquad \text{on } \, \Omega.
	\end{equation}
	In particular, if $\Omega = M$,
	\[
	\sqrt{1+|Du|^2} \le A e^{Cu} \qquad \text{on } \, M.
	\]	
\end{theorem}

\begin{remark}
	\emph{Because of the non-existence result of Heinz \cite{heinz}, Chern-Flanders \cite{chern,flanders} and Salavessa \cite{salavessa} for entire CMC graphs with nonzero $H$ on manifolds with subexponential volume growth, if $H \neq 0$ the bound for solutions on the entire $M$ is meaningful only if $\kappa > 0$. 
	}
\end{remark}

\begin{remark}
	Such $C \ge 0$ and $A \ge 1$ exist for all values of $\kappa \geq 0$, $H \in \R$. In particular,
	\begin{enumerate}
		\item [1).] if $\kappa = 0$ then we can choose
		$$
		A = 1 \qquad \text{and} \qquad C = 0
		$$
		for any value of $H\in\R$, and \eqref{CMC-gen-bound} becomes
		$$
		\sup_\Omega \sqrt{1+|Du|^2} = \limsup_{x\to\partial\Omega} \sqrt{1+|Du(x)|^2}
		$$
		\item [2).] if $\kappa>0$ and $H \le 0$, then we can choose
		\begin{equation}\label{eq_boundsAC}
		\begin{array}{ll}	
		A = 1, \; C=\sqrt{(m-1)\kappa^2-mH^2} & \quad \text{if } \, H > -\kappa\sqrt{\frac{m-1}{m}}, \\[0.4cm]
		A=1, \; C=0 & \quad \text{if } \, H \le -\kappa\sqrt{\frac{m-1}{m}};
		\end{array}
		\end{equation}
		\item [3).] if $\kappa>0$ and $H \geq 0$, we can choose
		$$
		A = 1 + \sqrt{\frac{mH}{\kappa\sqrt{m-1}}} \qquad \text{and} \qquad C = A\kappa\sqrt{m-1}.
		$$
	\end{enumerate}
	Note that, in the limit $H \ra 0$, the above choices allow us to recover inequality \eqref{eq_bellissima} for minimal graphs.
\end{remark}

It is interesting to compare our estimates with the known restrictions on the mean curvature of entire CMC graphs. By \cite{salavessa}, integrating on an open set $\Omega\Subset M$ the equation for the prescribed mean curvature, one has
\[
mH|\Omega| = \int_\Omega \diver \left( \frac{Du}{\sqrt{1+|Du|^2}}\right) = \int_{\partial \Omega} \frac{(Du,\eta)}{\sqrt{1+|Du|^2}} \le |\partial \Omega|.
\]
in particular, up to replacing $u$ with $-u$, $m|H|$ does not exceed the Cheeger constant of $M$:	
\[
m|H| \le c(M) : = \inf_{\Omega \Subset M} \frac{|\partial \Omega|}{|\Omega|}
\]
By Cheeger's inequality \cite{cheeger}, the bottom of the spectrum of $\Delta$ on $M$ satisfies $\lambda_1(M) \ge c(M)^2/4$. On the other hand, Brooks upper bound  \cite{brooks} for $\lambda_1(M)$ on manifolds satisfying $\Ricc \ge -(m-1)\kappa^2$, guarantees that
\[
\lambda_1(M) \le \frac{(m-1)^2\kappa^2}{4}.
\]
Summarizing, if $\Ricc \ge -(m-1)\kappa^2$ then any entire cmc graph must satisfy 
\begin{equation}\label{eq_cheeger}
|H| \le \frac{m-1}{m} \kappa.
\end{equation}
In particular, the second case in \eqref{eq_boundsAC} cannot happen for entire graphs. The bound in \eqref{eq_cheeger} is optimal since, in the hyperbolic space $\HH^m_\kappa$ of curvature $-\kappa^2$, Salavessa in \cite{salavessa} constructed a positive, radial, entire graph with mean curvature $H$ for each choice of $H \in (0, \kappa \frac{m-1}{m}]$. Furthermore, the graph has bounded gradient if $H < \kappa \frac{m-1}{m}$.

\section{Splitting for capillary graphs}

In the previous sections, we gave conditions to ensure that solutions of \eqref{problems} satisfy $f(u) \equiv 0$, and in some cases that they are constant. In this section, we give conditions to guarantee that the existence of a non-constant $u : M \ra \R$ solving the prescribed mean curvature problem 
\begin{equation}\label{eq_cap_split}
\diver \left( \frac{Du}{\sqrt{1+|Du|^2}} \right) = f(u)
\end{equation}
forces $M$ to split as a product $N \times \R$, with $u$ only depending on the $\R$-variable. The use of solutions of PDEs to prove splitting type results is classical in Geometric Analysis, and dates back at least to the theorem of Cheeger-Gromoll \cite{cheeger_gromoll} for manifolds with $\Ricc \ge 0$ and more than one end. In the same period, Schoen and Yau \cite{schoenyau} investigated the topology of complete, stable minimal hypersurfaces without boundary $\Sigma^m \hookrightarrow \bar M^{m+1}$ in a complete manifold with nonnegative sectional curvature, and discovered the following integral inequality for harmonic functions $u : \Sigma \ra \R$:
\begin{equation}\label{eq_schoenyau}
\frac{1}{m} \int_\Sigma \|\II_\Sigma\|^2 \varphi^2 + \frac{1}{m-1} \int_\Sigma \big\| \nabla \|\nabla u\|  \big\|^2 \varphi^2 \le \int_\Sigma \|\nabla \varphi\|^2 \|\nabla u\|^2 \qquad \forall \, \varphi \in \lip_c(\Sigma)
\end{equation}
(see also \cite{liwang_crelle} for another proof). The formula was later used by Li and Wang \cite{liwang_crelle} to prove that any such hypersurface must split as a Riemannian product $N \times \R$, provided that it is properly immersed and has more than one end. The elegant techniques were based on the theory of harmonic functions developed in \cite{litam_JDG}, and relate to their previous splitting theorems for manifolds with $\Ricc \ge -(m-1)\kappa^2$ and positive spectrum \cite{liwang2,liwang}.\par 
In \cite{far_HDR}, motivated by De Giorgi and Gibbons' conjectures for the Allen-Cahn equation
\[
\Delta u = u^3-u,
\]	
Farina independently discovered an integral formula, tightly related to \eqref{eq_schoenyau}, valid for each quasilinear operator. He used the integral inequality to establish the 1D-symmetry of certain stable solutions $u : \R^m \ra \R$ of 
\[
\diver \left( \frac{\varphi(|Du|)}{|Du|}Du \right) = f(u).	
\]
Here, stability means the the linearized operator at $u$ is non-negative, and $u$ is said to have 1D-symmetry if there exists $\omega \in \mathbb{S}^{m-1}$ such that $u(x) = \bar u(\omega \cdot x)$. However, he didn't publish the results up until 2008, in a joint paper with Sciunzi and Valdinoci \cite{fsv}. For the mean curvature operator, the formula reads as follows \cite[Thm. 2.5]{fsv}:
\begin{equation}\label{eq_farina}
\int_{\R^m} \left[ \frac{|D_\top|Du||^2}{(1+|Du|^2)^{3/2}} + \frac{|A|^2|Du|^2}{\sqrt{1+|Du|^2}} \right]\varphi^2 \le \int_{\R^m} \frac{|Du|^2|D\varphi|^2}{\sqrt{1+|Du|^2}} \qquad \forall \, \varphi \in \lip_c(\R^m),
\end{equation}
where, at a point $x$ such that $Du(x) \neq 0$, $D_\top$ and $A$ are the gradient and second fundamental form of the level set $\{u = u(x)\}$. They deduced the next theorems:	
\[
\begin{array}{l}
\text{if $f \in \lip_\loc(\R)$, stable solutions $u \in C^2(\R^m)$ of \eqref{eq_cap_split} have 1D-symmetry if either} \\[0.2cm]
\begin{array}{l}
\text{\cite[Thm.1.4]{fsv} $m = 2$ and either $f \ge 0$ or $f \le 0$ on $\R$;} \\[0.1cm]
\text{\cite[Thm.1.1]{fsv} $m=2$ and $|Du| \in L^\infty(\R^2)$;}\\[0.1cm]
\text{\cite[Thm.1.2]{fsv} $m=3$, $|Du| \in L^\infty(\R^3)$ and $\partial_{x_3} u > 0$;}\\[0.1cm] 
\end{array}
\end{array}
\]

\begin{remark}
	\emph{For nonconstant $f$, it would be desirable to guarantee the bound $|Du| \in L^\infty$ under suitable conditions on $u$ alone.
	}
\end{remark}
The study of capillary graphs in $\R \times M$ motivates the introduction of boundary conditions in \eqref{eq_cap_split}, and provide a natural setting where the results in the previous sections apply. A capillary hypersurface $\Sigma$ in a Riemannian manifold with boundary $\bar M$ is a CMC hypersurface that meets $\partial \bar M$ at a constant angle. If the angle is $\pi/2$, $\Sigma$ is said to have free boundary in $\bar M$. Capillarity hypersurfaces naturally have a variational interpretation, see \cite{finn_book,ros_souam}. Fix $\gamma \in (0, \pi/2]$, and suppose that $\Sigma$ is an embedded, connected hypersurface of $\bar M$, with boundary $\partial \Sigma \subset \partial \bar M$ that we assume to be transversal to $\partial \bar M$. Fix a connected, relatively compact open set $U \Subset \bar M$ in such a way that $\Sigma$ divides $U$ into two connected components $A$ and $B$. Let $B$ be the component inside which the angle $\gamma$ is computed. Note that $\partial \overline{B} \cap U$ is the union of a portion of $\Sigma$ and, possibly, of a relatively open subset $\Omega \subset \partial \bar M$. Then, $\Sigma$ is a stationary point for the functional
	\[
	|\Sigma \cap U| - \cos \gamma |\Omega|, 
	\]
with respect to variations that are compactly supported in $U$ and preserve the volumes of $A$ and $B$, if and only if $\Sigma$ is a CMC hypersurface satisfying 
	\[
	\langle \eta, \bar \eta \rangle = \cos \gamma \qquad \text{on } \, \partial \Sigma \cap U, 
	\]
where $\bar \eta$ and $\eta$ are, respectively, the outward pointing unit normals to $\partial \Sigma \hookrightarrow \Omega$ and $\partial \Sigma \hookrightarrow\Sigma$. \par
In \cite{finn_book}, Finn studied in depth various aspects of the capillarity problem. However, to the best of our knowledge, rigidity issues have been investigated only in compact ambient spaces $\bar M$ with a large amount of symmetries. Especially, capillary hypersurfaces in the unit ball $\bar M = \mathbb{B}^{m+1}$ attracted the attention of researchers, also in view of the link between free boundary minimal hypersurfaces in $\mathbb{B}^{m+1}$ and the Steklov eigenvalue problem (cf. \cite{fraserschoen_1,fraserschoen_2} and the references therein). We mention that stable capillary hypersurfaces in $\mathbb{B}^{m+1}$ have been classified in \cite{ros_vergasta,nunes,barbosa,lixiong} (free boundary case), and \cite{ros_souam,wangxia}, to which we refer for a detailed account. On the other hand, stable capillary hypersurfaces in $\R^{m+1}$ with planar boundaries were treated in \cite{lopez_wedge, lixiong_2}. We here address the case of noncompact graphs, so we let $\bar M = \Omega \times \R^+_0$ for some smooth open domain $\Omega \subset M$, and we consider $\Sigma$ to be the graph of $u : \Omega \ra \R^+_0$ satisfying $u = 0$ on $\partial \Omega$. Having fixed $U \Subset \bar M$, we let $B$ denote the subgraph of $\Sigma$ in $U$. Then, $\Sigma$ is a capillary hypersurface if and only if $u$ satisfies the overdetermined boundary value problem
\begin{equation}\label{eq_overdet}
\left\{ \begin{array}{ll}
\diver \left( \dfrac{Du}{\sqrt{1+|Du|^2}} \right) = mH & \quad \text{on } \, \Omega \\[0.5cm]
u = 0, \ \ \partial_\eta u = - \tan \gamma & \quad \text{on } \, \partial \Omega,
\end{array}\right.
\end{equation}
for some constant $H \in \R$, with $\eta$ the unit exterior normal to $\partial \Omega$. Problems like \eqref{eq_overdet} stimulated a vast research over the past 30 years, originally in the semilinear case  
\begin{equation}\label{eq_overdet_lapla}
\left\{ \begin{array}{l}
\Delta u + f(u) = 0 \qquad \text{on } \, \Omega \\[0.2cm]
u > 0 \qquad \text{on } \, \Omega, \\[0.2cm]
u = 0, \ \ \partial_\eta u = \mathrm{const} \qquad \text{on } \, \partial \Omega
\end{array}\right.
\end{equation}
with $f \in \lip_\loc(\R)$, a major issue being to characterize domains supporting a bounded solution of \eqref{eq_overdet_lapla} (that, agreeing with the literature, we call $f$-extremal domains). Rigidity for relatively compact $\Omega$ follows from pioneering works of Serrin and Weinberger \cite{serrin_movingplane,weinberger}, while the noncompact case has been investigated in the influential \cite{bcn_1,bcn_2} by Berestycki, Caffarelli and Nirenberg. In \cite{bcn_1}, the authors conjectured that the only connected, smooth $f$-extremal domains of $\R^m$ with connected complementary are either the ball, the half-space, the cylinder $B^k \times \R^{m-k}$ or their complements. The conjecture turns out to be false in full generality, by counterexamples in \cite{sicbaldi} ($m \ge 3$) and \cite{ros_ruiz_sic_2} ($m = 2$ and $\Omega$ an exterior region). Surprisingly, in $\R^2$ the conjecture is true if both $\Omega$ and its complementary are unbounded \cite{ros_ruiz_sic_1}. Corresponding rigidity results for $f$-extremal domains in $\HH^2$ and $\Sph^2$ have been obtained in \cite{espi_far_maz,espi_maz}.

Typically, the characterization of domains $\Omega \subset \R^m$ supporting a solution of \eqref{eq_overdet_lapla} proceeds by first showing that $u$ is monotone in one variable, namely $\partial_m u > 0$, which is the case  for suitable $f$ (for instance, derivatives of bistable nonlinearities, \cite{bcn_1}). Since $\partial_m u$ satisfies the linearized equation $\Delta w + f'(u)w = 0$, its positivity implies that $u$ is stable. In \cite{far_vald_ARMA}, Farina and Valdinoci made the remarkable discovery that the identity corresponding to \eqref{eq_farina} for $\Delta u + f(u) = 0$, meaning 
\begin{equation}\label{eq_farina_2}
\int_{\Omega} \left[ |D_\top|Du||^2 + |A|^2|Du|^2 \right]\varphi^2 \le \int_{\Omega} |Du|^2|D\varphi|^2,
\end{equation}
remains true even if the support of $\varphi$ contains a portion of $\partial\Omega$, provided that $u$ satisfies \eqref{eq_overdet_lapla} and $\partial_m u > 0$. In other words, the contributions of boundary terms to \eqref{eq_farina_2} on $\partial \Omega \cap {\rm spt } \varphi$ vanish identically if both $u$ and $\partial_\eta u$ are constant on $\partial \Omega$, a fact that enables to use \eqref{eq_farina_2} to characterize domains supporting a solution of \eqref{eq_overdet_lapla}. This point of view has been further developed in \cite{fmv} on manifolds with $\Ricc \ge 0$, to obtain splitting results for domains supporting a solution of \eqref{eq_overdet_lapla} that is monotone in the direction of a Killing field. 

In the literature, overdetermined problems for more general operators have mainly been investigated in the $p$-Laplacian case, leaving the corresponding questions for the mean curvature operator, to the best of our knowledge, mostly unexplored. Taking into account that \eqref{eq_overdet} alone imposes restriction on the geometry of $\Omega$ even without overdetermined boundary conditions (cf. \cite{lopez,collin_krust,imperapigolasetti}, and the survey in \cite{jfwang}), we could expect more rigidity than in the case of the Laplacian. However, as far as we know there is still no attempt to study the equivalent of the Berestycki-Caffarelli-Nirenberg conjecture for \eqref{eq_cap_split}. We here comment on the very recent \cite{cmmr_2}, where we address this conjecture for constant $f$, proving the following splitting theorem for capillary graphs on manifolds with non-negative Ricci curvature. The result also relates to the work in progress \cite{far_fra_mar}, where the authors study the problem for general $f$, with emphasis on bistable nonlinearities. To state the theorem, we recall that a manifold with boundary $(N,h)$ is said to be parabolic if, for each compact subset $K \subset N$, the capacity
\[
\capac(K) : = \inf \left\{ \int_N |D\varphi|^2 \di x \ : \ \varphi \in \lip_c(N), \varphi \ge 1 \ \text{on } \, K \right\}
\]
vanishes. It can be checked that the condition corresponds to the recurrency of the minimal Brownian motion on $N$ that is normally reflected on $\partial N$. By \cite{imperapigolasetti}, the parabolicity of $(N,h)$ is equivalent to the validity of the following principle: 
\[
\left\{\begin{array}{l}
\Delta_h u \ge 0 \quad \text{on } \, \mathrm{Int}(N) \\[0.2cm]
\partial_\eta u \le 0 \quad \text{on } \, \partial N, \\[0.2cm]
\sup_N u < \infty 
\end{array}\right. \qquad \Longrightarrow \qquad \sup_N u = \sup_{\partial N} u,
\]
with $\eta$ the outward pointing normal to $\partial N$. If $N= \overline\Omega$ and $\Omega$ is an open subset of a complete manifold $M$, by \cite{imperapigolasetti} a sufficient condition for $\overline{\Omega}$ to be parabolic is that 
\begin{equation}\label{eq_suffipara}
\int^{\infty} \frac{\di s}{|\partial B_s \cap \Omega|} = \infty,
\end{equation}
with $B_s$ the metric ball of radius $s$ in $M$ centered at a fixed origin of $M$. By an application of H\"older inequality (cf. \cite[Prop. 1.3]{rigolisetti}), \eqref{eq_suffipara} is implied by 
\[
\int^{\infty} \frac{s \, \di s}{|B_s \cap \Omega|} = \infty. 
\]

%

\begin{theorem}\label{teo_splitting}
	Let $(M, ( \, , \, ) )$ be a complete Riemannian manifold of dimension $m \ge 2$, and let $\Omega \subset M$ be a connected, parabolic open subset with smooth boundary. Assume that 
	\[
	\left\{ \begin{array}{ll}
	\Ricc \ge - (m-1)\kappa^2 & \quad \text{on } \, M, \\[0.2cm]
	\Ricc \ge 0 & \quad \text{on } \, \Omega.
	\end{array}\right. 
	\]
	for some constant $\kappa>0$. Split $\partial \Omega$ into its connected components $\{\partial_j\Omega\}$, $1 \le j \le j_0$, possibly $j_0 = \infty$. Let $u \in C^2(\overline\Omega)$ be a non-constant solution of the capillarity problem
	\begin{equation}\label{eq_overdet_cmc}
	\left\{ \begin{array}{ll}
	\diver \left(\dfrac{Du}{\sqrt{1+ |Du|^2}} \right) = mH \qquad \text{on } \, \Omega, \\[0.5cm]
	\inf_\Omega u > -\infty, \\[0.2cm]
	u = b_j, \ \partial_{\eta} u = c_j \quad \text{on } \, \partial_j \Omega,
	\end{array}\right.
	\end{equation}
	for some set of constants $H, c_j, b_j \in \R$, with $\eta$ the exterior normal to $\partial \Omega$, $\{c_j\}$ a bounded sequence if $j_0 = \infty$, and with the agreement $b_1 \le b_2 \le \ldots \le b_{j_0}$. Assume that either
	\[
	\text{$u$ is constant on $\partial \Omega$},
	\]
	or that 
	\begin{equation}\label{ipo_riccivol}
	\liminf_{r \ra \infty} \frac{\log |\partial \Omega \cap B_r|}{r^2}< \infty. 
	\end{equation}
	If there exists a Killing vector field $X$ on $\overline\Omega$ with the following properties:
	\begin{equation}\label{ipo_X}	
	\sup_\Omega |X| < \infty,  \qquad \left\{ \begin{array}{l}
	c_j (X, \eta) \ge 0 \quad \text{on $\partial_j \Omega$, for every $j$,} \\[0.2cm]
	c_j (X, \eta) \not \equiv 0 \quad \text{on $\partial_j \Omega$, for some $j$,}
	\end{array}\right.  
	\end{equation}
	then: 
	\begin{itemize}
		\item[(i)] $\Omega = (0, T) \times N$ with the product metric, for some $T \le \infty$ and some complete, boundaryless $N$ with $\Ricc_N \ge 0$. 
		\item[(ii)] The product $(X, \partial_t)$ is a positive constant on $\Omega$. 
		\item[(iii)] The solution $u(t,x)$ only depends on the variable $t \in (0,T)$. Moreover, setting $\partial_1 \Omega = \{0\} \times N$, 	
		\begin{itemize}
			\item[-] if $H=0$, then $c_1 < 0$ and 
			\begin{equation}\label{eq_split_minimal}
			u(t) = b_1 - c_1 t \qquad \text{on } \, \Omega.
			\end{equation}
			\item[-] if $H>0$, then $c_1 \le 0$ and 
			\begin{equation}\label{eq_split_CMC}
			u(t) = b_1 + \frac{1}{mH} \left( \frac{1}{\sqrt{1+c_1^2}} - \sqrt{1- \left(mHt -\frac{c_1}{\sqrt{1+c_1^2}} \right)^2} \right), 
			\end{equation}
			in particular, 
			\[
			\disp T < \frac{1}{mH}\left( 1 + \frac{c_1}{\sqrt{1+c_1^2}}\right) \ \ \ \text{if } \, H > 0, \qquad T \le \frac{1}{m|H|}\frac{|c_1|}{\sqrt{1+c_1^2}} \ \ \ \text{if } \, H< 0.
			\]	
		\end{itemize}
	\end{itemize}
	Furthermore, if $u$ is constant on the entire $\partial \Omega$, then the requirement \eqref{ipo_riccivol} can be omitted and the case $H=0$ in (iii) occurs.
\end{theorem}

\begin{remark}
	\emph{Evidently, if $T < \infty$ then \eqref{eq_split_minimal}, \eqref{eq_split_CMC} relate the constants $b_2,c_2$ describing the boundary data on $\partial_2 \Omega = \{T\} \times N$ to $b_1,c_1$.
	}
\end{remark}

We stress that $X$ is just required to satisfy a mild relation, to make it coherent with the Neumann data: precisely, $X$ shall point outwards on components $\partial_j \Omega$ for which the outward pointing normal derivative of $u$ is positive (i.e., $c_j>0$), leaving the possibility that $X$ be tangent to $\partial_j \Omega$, while $X$ shall point inwards if $c_j<0$. No condition is imposed on components of $\partial \Omega$ where $c_j = 0$. However, to avoid a degenerate case, we require that $X$ is not tangent to the entire portion of $\partial \Omega$ where $\partial_\eta u \neq 0$.\par
To conclude, we briefly describe the proof of Theorem \ref{teo_splitting}. The argument follows the lines of \cite[Thm.5]{fmv}, with some notable improvements depending on the gradient estimate in Theorem \ref{thm-CMCbound}, and can be divided into the following steps.
\begin{itemize}
	\item[-] We prove the monotonicity of $u$ along the flow lines of $X$, namely, that $\bar v : = (Du,X)$ is positive on $\Omega$. To this aim, we first observe that $\LL_g \bar v = 0$ on $\Omega$, with $\LL_g$ the weighted operator in \eqref{def_Lg}. The boundedness of $|X|$ and the gradient estimate \eqref{CMC-gen-bound}, applied with $\kappa=0$, imply $\inf_\Omega \bar v > -\infty$, while $\bar v \ge 0, \not \equiv 0$ on $\partial \Omega$ because of our boundary condition \eqref{ipo_X}. Next, because of the boundedness of $|Du|$ and the parabolicity of $(\overline{\Omega}, \sigma)$, the weighted operator $\LL_g$ is parabolic on $\Sigma = (\overline{\Omega},g)$ (cf. \cite{imperapigolasetti}, and \cite{AMR_book,prsmemoirs}), and therefore 
	\[
	\inf_\Omega \bar v \equiv \inf_{\partial \Omega} \bar v \ge 0.
	\]
	The conclusion $\bar v > 0$ follows from the strong maximum principle. In particular, $\di u \neq 0$ on $\Omega$.
	\item[-] We show the validity of the following integral formula on $\Sigma$: for every $\varphi \in \lip_c(\overline\Sigma)$,
	\begin{equation}\label{form_good}
	\begin{array}{l}
	\disp \int_\Sigma \left[ W^2\left( \|\nabla_\top \|\nabla u\|\|^2 + \|\nabla u\|^2 \|A\|^2\right) + \frac{\Ricc(Du,Du)}{1+|Du|^2} \right]\varphi^2 \\[0.5cm]
	\qquad \disp + \int_\Sigma \frac{\bar v^2}{W^2} \left\| \nabla \left( \frac{\varphi\|\nabla u\|W}{\bar v} \right) \right\|^2 \le \int_\Sigma \|\nabla u\|^2 \|\nabla \varphi \|^2
	\end{array}
	\end{equation}
	where, at a point $x$ such that $\di u(x) \neq 0$, $\nabla_\top$ and $A$ are the gradient and second fundamental form of the level set $\{u = u(x)\}$ in $\Sigma$, and integration is with respect to the volume measure of $\Sigma$. Inequality \eqref{form_good} is tightly related to \eqref{eq_schoenyau} and \eqref{eq_farina_2}; in particular, the ``magic cancellation" of the boundary terms in \eqref{eq_farina_2} still holds for \eqref{eq_overdet_cmc}. As shown in \cite{far_fra_mar}, in fact, overdetermined problems associated to a large class of quasilinear operators lead to integral formulas like \eqref{form_good}.
	\item[-] We use again the gradient bound $|Du| \le C$ to deduce that 
	\[
	\int_\Sigma \|\nabla u\|^2 \|\nabla \varphi \|^2 \le \int_\Omega W|D\varphi|^2 \di x \le C' \int_\Omega |D\varphi|^2 \di x, 
	\]
	for some constant $C'$. Therefore, the parabolicity of $(\overline{\Omega},\sigma)$ enables to deduce the existence of a sequence $\varphi_j \uparrow 1$ such that
	\[
	\int_\Omega |D\varphi_j|^2 \di x \ra 0,
	\]
	so taking limits we deduce 
	\[
	\|\nabla_\top \|\nabla u\|\|^2 + \|\nabla u\|^2 \|A\|^2 + \left\| \nabla \left( \frac{\|\nabla u\|W}{\bar v}\right) \right\|^2 \equiv 0 \quad \text{on } \, \Omega.
	\]		
	By adapting \cite{fmv}, the first two conditions guarantee that $\Omega$ splits as a Riemannian product along the flow lines of $\nabla u$ (everywhere nonvanishing on $\Omega$), and that $u$ only depends on the split direction. Also, the constancy of $\|\nabla u\|W/\bar v$ rewrites as the constancy of $(X, \partial_t)$ on $\Omega$.
\end{itemize}

\section*{Acknowledgements}
P. Pucci  is a member of the {\em Gruppo Nazionale per
	l'Analisi Ma\-te\-ma\-ti\-ca, la Probabilit\`a e le loro Applicazioni}
(GNAMPA) of the {\em Istituto Nazionale di Alta Matematica} (INdAM).
The manuscript was realized within the auspices of the INdAM -- GNAMPA Projects
{\em Equazioni alle derivate parziali: problemi e mo\-del\-li} (Prot\_U-UFMBAZ-2020-000761).
P. Pucci was also partly supported  by  the {\em Fondo Ricerca di Base di Ateneo -- Eser\-ci\-zio 2017--2019} of the University of Perugia, named {\em PDEs and Nonlinear Analysis}.

\section*{Conflict of interest}

The authors declare that they have no conflict of interest.



\end{document}